%% file: Ising_armexponents.tex
\pdfoutput=1
\documentclass[12pt]{paper}

\usepackage{color}
\usepackage{helvet}         
\usepackage{courier}        
\usepackage{type1cm}        
\usepackage{makeidx}         
\usepackage{graphicx}        
\usepackage{multicol}        
\usepackage[bottom]{footmisc}
\usepackage{amsmath}
\usepackage{amssymb}
\usepackage{bbold}
\usepackage{amsthm}
\usepackage{subcaption}
\usepackage{comment}
\newtheorem{theorem}{Theorem}
\newtheorem{corollary}[theorem]{Corollary}

\newtheorem{lemma}[theorem]{Lemma}

\newtheorem{proposition}[theorem]{Proposition}

\newtheorem{remark}[theorem]{Remark}

\usepackage[top=2cm, bottom=2cm, left=2cm, right=2cm]{geometry}
\numberwithin{theorem}{section}
\numberwithin{figure}{section}
\numberwithin{equation}{section}

\DeclareMathOperator{\dist}{dist}
\DeclareMathOperator{\SLE}{SLE}

\DeclareMathOperator{\free}{free}

\begin{document}

\title{Alternating Arm Exponents for the Critical Planar Ising Model}
\author{Hao Wu}


%
%
\maketitle

\abstract{We derive the alternating arm exponents of critical Ising model. We obtain six different patterns of alternating boundary arm exponents which correspond to the boundary conditions $(\ominus\oplus)$, $(\ominus\free)$ and $(\free\free)$, and the alternating interior arm exponents.\\
\noindent\textbf{Keywords:} critical Ising model, Schramm Loewner Evolution, arm exponents. }
\newcommand{\eps}{\epsilon}
\newcommand{\ov}{\overline}
\newcommand{\U}{\mathbb{U}}
\newcommand{\T}{\mathbb{T}}
\newcommand{\HH}{\mathbb{H}}
\newcommand{\LA}{\mathcal{A}}
\newcommand{\LB}{\mathcal{B}}
\newcommand{\LC}{\mathcal{C}}
\newcommand{\LD}{\mathcal{D}}
\newcommand{\LF}{\mathcal{F}}
\newcommand{\LK}{\mathcal{K}}
\newcommand{\LE}{\mathcal{E}}
\newcommand{\LG}{\mathcal{G}}
\newcommand{\LL}{\mathcal{L}}
\newcommand{\LM}{\mathcal{M}}
\newcommand{\LQ}{\mathcal{Q}}
\newcommand{\LU}{\mathcal{U}}
\newcommand{\LV}{\mathcal{V}}
\newcommand{\LZ}{\mathcal{Z}}
\newcommand{\LH}{\mathcal{H}}
\newcommand{\R}{\mathbb{R}}
\newcommand{\C}{\mathbb{C}}
\newcommand{\N}{\mathbb{N}}
\newcommand{\Z}{\mathbb{Z}}
\newcommand{\E}{\mathbb{E}}
\newcommand{\PP}{\mathbb{P}}
\newcommand{\QQ}{\mathbb{Q}}
\newcommand{\A}{\mathbb{A}}
\newcommand{\one}{\mathbb{1}}
\newcommand{\bn}{\mathbf{n}}
\newcommand{\MR}{MR}
\newcommand{\cond}{\,|\,}
\newcommand{\la}{\langle}
\newcommand{\ra}{\rangle}
\newcommand{\tree}{\Upsilon}

\section{Introduction}
\input{tex/introduction}
\section{Preliminaries on SLE}\label{sec::sle_pre}
\input{tex/sle_preliminaries}
\section{SLE Boundary Arm Exponents}
\label{sec::sle_boundary}
\subsection{Definitions and Statements}
\input{tex/sle_statements}
\subsection{Estimates on the derivatives}
\input{tex/sle_boundary_derivatives}
\subsection{Proofs of Propositions \ref{prop::sle_boundary_alpha} and \ref{prop::sle_boundary_beta}}
\input{tex/sle_boundary1}
\subsection{Proof of Proposition \ref{prop::sle_boundary_gamma}}
\input{tex/sle_boundary2}
\section{SLE Interior Arm Exponents}\label{sec::sle_interior}
\input{tex/sle_interior}
\section{Ising Model}\label{sec::ising}
\subsection{Definitions and Properties}\label{subsec::ising_properties}
\input{tex/ising}

\subsection{Quasi-Multiplicativity}
\input{tex/qm}
\subsection{Proofs of Theorems \ref{thm::ising_boundary} and \ref{thm::ising_interior}}
\input{tex/ising_interface_cvg}

\input{tex/ising_arms}

\section{Appendix: One-point estimate of the intersection of $\SLE_{\kappa}(\rho)$ with the boundary}
\input{tex/appendix_onepoint_sle}

\bibliographystyle{alpha}
\bibliography{bibliography}
\smallbreak
\noindent Hao Wu\\
\noindent NCCR/SwissMAP, Section de Math\'{e}matiques, Universit\'{e} de Gen\`{e}ve, Switzerland\\
\noindent\textit{and} Yau Mathematical Sciences Center, Tsinghua University, China\\
\noindent hao.wu.proba@gmail.com

\end{document}

%% file: tex/introduction.tex
The Lenz-Ising model is one of the simplest models in statistical physics. It is a model on the spin configurations. Each vertex $x$ has a spin $\sigma_x$ which is $\oplus$ or $\ominus$. Each configuration of spins $\sigma=(\sigma_x, x\in V)$ has an intrinsic energy---the Hamiltonian:
\[H(\sigma)=-\sum_{x\sim y}\sigma_x\sigma_y.\] 
A natural way to sample the random configuration is the Boltzman measure: 
\[\mu[\sigma]\propto \exp\left(-\frac{1}{T}H(\sigma)\right),\]
where $T$ is the temperature. This measure favors configurations with low energy. Due to recent celebrated work of Chelkak and Smirnov \cite{ChelkakSmirnovIsing, CDCHKSConvergenceIsingSLE}, it is proved that at the critical temperature, the interface of Ising model is conformally invariant and converges to a random curve---Schramm Loewner Evolution ($\SLE_3$). In this paper, we drive the alternating arm exponents of critical Ising model. 

An arm is a simple path of $\oplus$ or of $\ominus$. We are interested in the decay of the probability that there are a certain number of arms of certain pattern in the semi-annulus $A^+(n,N)$ or annulus $A(n,N)$ connecting the inner boundary to the outer boundary. This probability should decay like a power in $N$ as $N\to \infty$, and the exponent in the power is called the critical arm exponents. 

In \cite{LawlerSchrammWernerExponent1, LawlerSchrammWernerExponent2, LawlerSchrammWernerExponent3, LawlerSchrammWernerOneArmExponent, SmirnovWernerCriticalExponents}, the authors derived the value of the arm exponents for critical percolation; in \cite{WuPolychromaticArmFKIsing}, the author derived the value of the arm exponents for critical FK-Ising model. As explained in \cite{SmirnovWernerCriticalExponents}, the strategy to derive the arm exponents is the following: one needs three inputs: (1) the convergence of the interface to $\SLE$; (2) the arm exponents of $\SLE$; and (3) the quasi-multiplicativity. This strategy also works for the critical Ising model. In this paper, we derive the boundary arm exponents and the interior arm exponents of $\SLE_{\kappa}$ and its variant $\SLE_{\kappa}(\rho)$, and then explain how to apply these formulae to get  the alternating arm exponents of critical Ising model.

\begin{theorem}\label{thm::ising_boundary}
For the critical planar Ising model on the square lattice, we have the following six different patterns of the boundary arm exponents (the arm patterns are in clockwise order). 
Fix $j\ge 1$. 
\begin{itemize}
\item Consider the boundary condition $(\oplus\oplus)$ and the arms pattern $(\ominus\oplus\ominus\cdots \oplus\ominus)$ with length $2j-1$. The boundary arm exponents for this pattern is given by 
\begin{equation}\label{eqn::ising_alpha_odd}
\alpha^{+}_{2j-1}=j(4j+1)/3.
\end{equation}
\item Consider the boundary condition $(\ominus\oplus)$ and the arms pattern $(\oplus\ominus\cdots\oplus\ominus)$ with length $2j$. The boundary arm exponents for this pattern is given by 
\begin{equation}\label{eqn::ising_alpha_even}
\alpha^+_{2j}=j(4j+5)/3.
\end{equation}
\item Consider the boundary condition $(\ominus \free)$ and the arms pattern $(\oplus\ominus\oplus\cdots\ominus\oplus)$ with length $2j-1$. The boundary arm exponents for this pattern is given by 
\begin{equation}\label{eqn::ising_beta_odd}
\beta^+_{2j-1}=2j(2j-1)/3.
\end{equation}
\item Consider the boundary condition $(\ominus \free)$ and the arms pattern $(\oplus\ominus\oplus\cdots\oplus\ominus)$ with length $2j$. The boundary arm exponents for this pattern is given by 
\begin{equation}\label{eqn::ising_beta_even}
\beta^+_{2j}=2j(2j+1)/3.
\end{equation}
\item Consider the boundary condition $(\free\free)$ and the arms pattern $(\ominus\oplus\ominus\cdots\ominus)$ with length $2j-1$. The boundary arm exponents for this pattern is given by 
\begin{equation}\label{eqn::ising_gamma_odd}
\gamma^+_{2j-1}=(2j-1)(4j-3)/6.
\end{equation}
\item Consider the boundary condition $(\free\free)$ and the arms pattern $(\ominus\oplus\cdots\ominus\oplus)$ with length $2j$. The boundary arm exponents for this pattern is given by 
\begin{equation}\label{eqn::ising_gamma_even}
\gamma^+_{2j}=j(4j-1)/3.
\end{equation}
\end{itemize}
\end{theorem}
\begin{figure}[ht!]
\begin{subfigure}[b]{0.3\textwidth}
\begin{center}
\includegraphics[width=\textwidth]{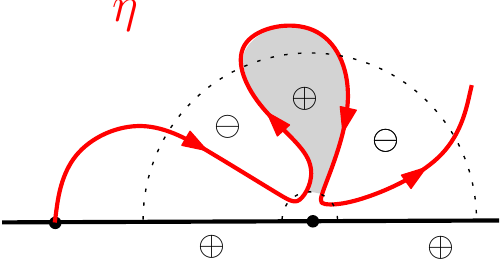}
\end{center}
\caption{$\alpha^+_3$: $(\ominus\oplus\ominus)$ with boundary condition $(\oplus\oplus)$.}
\end{subfigure}
$\quad$
\begin{subfigure}[b]{0.3\textwidth}
\begin{center}\includegraphics[width=\textwidth]{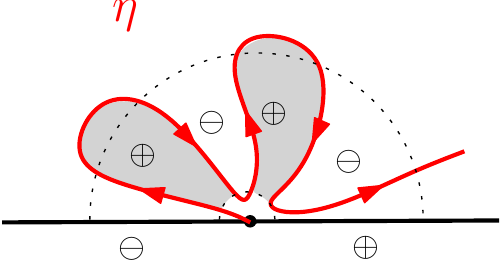}
\end{center}
\caption{$\alpha^+_4$: $(\oplus\ominus\oplus\ominus)$ with boundary condition $(\ominus\oplus)$.}
\end{subfigure}
$\quad$
\begin{subfigure}[b]{0.3\textwidth}
\begin{center}
\includegraphics[width=\textwidth]{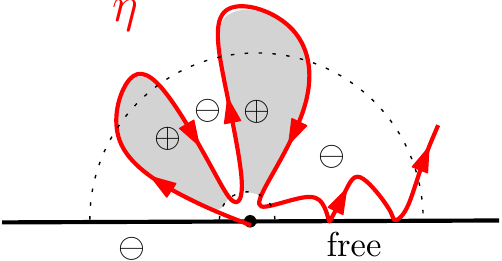}
\end{center}
\caption{$\beta^+_4$: $(\oplus\ominus\oplus\ominus)$ with boundary condition $(\ominus\free)$.}
\end{subfigure}
\begin{subfigure}[b]{0.3\textwidth}
\begin{center}
\includegraphics[width=\textwidth]{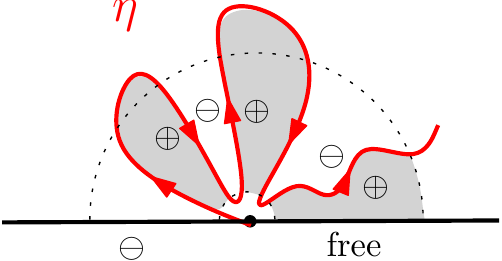}
\end{center}
\caption{$\beta^+_5$: $(\oplus\ominus\oplus\ominus\oplus)$ with boundary condition $(\ominus\free)$.}
\end{subfigure}
$\quad$
\begin{subfigure}[b]{0.3\textwidth}
\begin{center}\includegraphics[width=\textwidth]{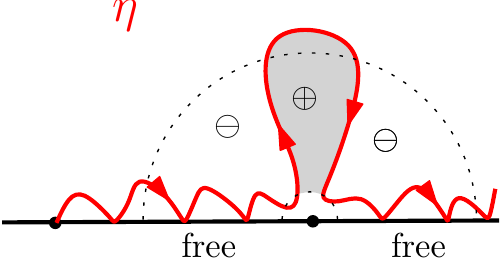}
\end{center}
\caption{$\gamma^+_3$: $(\ominus\oplus\ominus)$ with boundary condition $(\free\free)$.}
\end{subfigure}
$\quad$
\begin{subfigure}[b]{0.3\textwidth}
\begin{center}
\includegraphics[width=\textwidth]{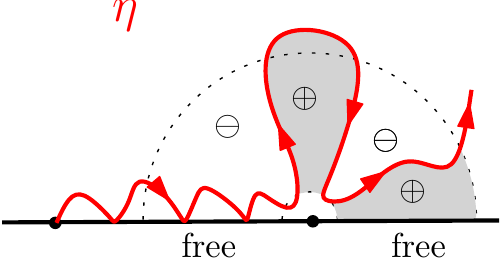}
\end{center}
\caption{$\gamma^+_4$: $(\ominus\oplus\ominus\oplus)$ with boundary condition $(\free\free)$.}
\end{subfigure}
\caption{\label{fig::boundary_arms} The six different patterns of boundary arm exponents in Theorem \ref{thm::ising_boundary}.}
\end{figure}

\begin{theorem}\label{thm::ising_interior}
For the critical planar Ising model on the square lattice, the alternating interior arm exponents with length $2j$ for $j\ge 1$ is given by 
\begin{equation}\label{eqn::ising_interior}
\alpha_{2j}=(16j^2-1)/24. 
\end{equation}
\end{theorem}

\begin{remark}
In Theorem \ref{thm::ising_boundary}, the arm exponent $\gamma_2^+=1$ is a universal arm exponent of critical Ising model. In other words, the fact that $\gamma_2^+=1$ can be obtained by standard proof of universal arm exponents using RSW---Proposition \ref{prop::ising_rsw}. 
\end{remark}

\begin{remark}
For the critical planar Ising model (on the square lattice) in a topological rectangle $(\Omega, a, b, c, d)$ with free boundary conditions, consider the probability that there exists a path of $\oplus$ connecting $(ab)$ to $(cd)$. It is proved in \cite{BenoistDuminilHonglerIsingFree} that, as the mesh-size goes to zero, this probability converges to a function $f$ which maps topological rectangles to $[0,1]$ and it is conformal invariant. Therefore, the limit of this probability only depends on the extremal distance of the rectangle. Whereas, the exact formula for $f$ is unknown. As a consequence of Theorem \ref{thm::ising_boundary}, we could give the asymptotics of this function $f$. Consider the rectangle $[0,\pi L]\times [0,1]$ and let $f(L)$ be the limit of the probability that the Ising model with free boundary conditions has a $\oplus$ horizontal crossing of the rectangle. Then we have 
\[f(L)=\exp(-L(1/6+o(1))).\]
\end{remark}

\smallbreak
\noindent\textbf{Relation to previous works.} The formulae (\ref{eqn::ising_alpha_odd}) and (\ref{eqn::ising_alpha_even}) are also obtained in \cite{WuZhanSLEBoundaryArmExponents}. In \cite{WuPolychromaticArmFKIsing}, the author derived the arm exponents of $\SLE_{\kappa}$ for $\kappa\in (4,8)$. In this paper, we derive the arm exponents for $\SLE_{\kappa}$ and $\SLE_{\kappa}(\rho)$ for $\kappa\in (0,4)$. The difficulty in this paper is that, when one estimates the arm events of $\SLE_{\kappa}(\rho)$, one has two more variables to take care of. The idea of the proof is similar to the one presented in \cite{WuZhanSLEBoundaryArmExponents, WuPolychromaticArmFKIsing}, but the increase in the number of variables causes certain technical difficulty. This difficulty is treated in Section \ref{sec::sle_boundary}. The boundary $1$-arm exponent $\gamma^+_1$ is related to the Hausdorff dimension of the intersection of $\SLE_{\kappa}(\rho)$ with the boundary which is $1-\gamma^+_1$. This dimension was obtained in \cite{WernerWuCLEtoSLE, MillerWuSLEIntersection}. Moreover, the formulae (\ref{eqn::ising_alpha_odd}) and (\ref{eqn::ising_interior}) were predicted by KPZ in \cite[Equations (11.42), (11.43)]{DuplantierFractalGeometry}.
\smallbreak
\noindent\textbf{Outline.} We give preliminaries on SLE in Section \ref{sec::sle_pre}. We derive the boundary arm exponents of $\SLE_{\kappa}(\rho)$ with $\kappa\in (0,4)$ and $\rho\in (-2, 0]$ in Section \ref{sec::sle_boundary}. We derive the interior arm exponents of $\SLE_{\kappa}$ with $\kappa\in (0,4)$ in Section \ref{sec::sle_interior}. Finally, we explain how to apply these formulae to obtain the alternating arm exponents of critical Ising in Section \ref{sec::ising} and complete the proof of Theorems \ref{thm::ising_boundary} and \ref{thm::ising_interior}. 
\smallbreak
\noindent\textbf{Acknowledgment.} The author is supported by the NCCR/SwissMAP, the ERC AG COMPASP, the Swiss NSF. The author acknowledges Hugo Duminil-Copin, Matan Harel, Christophe Garban, Gregory Lawler, Aran Raoufi, Stanislav Smirnov, Vincent Tassion, and David Wilson for helpful discussions. The author acknowledges Dapeng Zhan for helpful comments on the previous version of this paper.  

%% file: tex/sle_preliminaries.tex
\noindent\textbf{Notations.} We denote by $f\lesssim g$ if $f/g$ is bounded from above by universal finite constant, by $f\gtrsim g$ if $f/g$ is bounded from below by universal positive constant, and by $f\asymp g$ if $f\lesssim g$ and $f\gtrsim g$.

\noindent We denote by  
\[f(\eps)=g(\eps)^{1+o(1)}\quad \text{if}\quad \lim_{\eps\to 0}\frac{\log f(\eps)}{\log g(\eps)}=1.\]

\noindent For $z\in\C, r>0$, we denote
$B(z, r)=\{w\in\C: |w-z|<r\}$.

\noindent For two subsets $A, B\subset\C$, we denote $ \dist(A, B)=\inf\{|x-y|: x\in A, y\in B\}$. We assume that $\dist(A, \emptyset)=\infty$.

\noindent Let $\Omega$ be an open set and let $V_1, V_2$ be two sets such that $V_1\cap\overline{\Omega}\neq\emptyset$ and $V_2\cap\overline{\Omega}\neq\emptyset$. We denote the extremal distance between $V_1$ and $V_2$ in $\Omega$ by $d_{\Omega}(V_1, V_2)$, see \cite[Section 4]{AhlforsConformalInvariants} for the definition.

\subsection{$\HH$-hull and Loewner chain} 
We call a compact subset $K$ of $\overline{\HH}$ an $\HH$-hull if $\HH\setminus K$ is simply connected. Riemann's Mapping Theorem asserts that there exists a unique conformal map $g_K$ from $\HH\setminus K$ onto $\HH$ such that
\[\lim_{|z|\to\infty}|g_K(z)-z|=0.\]
We call such $g_K$ the conformal map from $\HH\setminus K$ onto $\HH$ normalized at $\infty$.
\begin{lemma}\label{lem::extremallength_argument}
Fix $x>0$ and $\eps>0$.
Let $K$ be an $\HH$-hull and let $g_K$ be the conformal map from $\HH\setminus K$ onto $\HH$ normalized at $\infty$. Assume that
\[x>\max(K\cap\R).\]
Denote by $\gamma$ the connected component of $\HH\cap (\partial B(x,\eps)\setminus K)$ whose closure contains $x+\eps$. Then $g_K(\gamma)$ is contained in the ball with center $g_K(x+\eps)$ and radius $3(g_K(x+3\eps)-g_K(x+\eps))$, hence it is also contained in the ball with center $g_K(x+3\eps)$ and radius $8\eps g_K'(x+3\eps)$.
\end{lemma}
\begin{proof}
This lemma is proved in \cite[Lemma 2.1]{WuZhanSLEBoundaryArmExponents}. To be self-contained, we repeat the proof here. 
Define $r^*=\sup\{|z-g_K(x+\eps)|: z\in g_K(\gamma)\}$.
It is sufficient to show
\begin{equation}\label{eqn::extremal_aux}
r^*\le 3(g_K(x+3\eps)-g_K(x+\eps)).
\end{equation}
We will prove (\ref{eqn::extremal_aux}) by estimates on the extremal distance:
\[d_{\HH}(g_K(\gamma), [g_K(x+3\eps), \infty)).\]
By the conformal invariance and the comparison principle \cite[Section 4.3]{AhlforsConformalInvariants}, we can obtain the following lower bound.
\begin{align*}
d_{\HH}(g_K(\gamma), [g_K(x+3\eps), \infty))&=d_{\HH\setminus K}(\gamma, [x+3\eps, \infty))\\
&\ge d_{\HH\setminus B(x,\eps)}(B(x,\eps), [x+3\eps, \infty))\\
&=d_{\HH\setminus \U}(\U, [3,\infty))=d_{\HH}([-1,0], [1/3,\infty)).
\end{align*}
On the other hand, we will give an upper bound. Recall a fact for extremal distance: for $x<y$ and $r>0$, the extremal distance in $\HH$ between $[y,\infty)$ and a connected set $S\subset\overline{\HH}$ with $x\in\overline{S}\subset\overline{B(x,r)}$ is maximized when $S=[x-r,x]$, see \cite[Chapter I-E, Chapter III-A]{AhlforsQuasiconformal}. Since $g_K(\gamma)$ is connected and $g_K(x+\eps)\in \R\cap\overline{g_K(\gamma)}$, by the above fact, we have the following upper bound.
\begin{align*}
d_{\HH}(g_K(\gamma), [g_K(x+3\eps), \infty))&\le d_{\HH}([g_K(x+\eps)-r^*, g_K(x+\eps)], [g_K(x+3\eps), \infty))\\
&=d_{\HH}\left([-r^*,0], \left[g_K(x+3\eps)-g_K(x+\eps),\infty\right)\right).
\end{align*}
Combining the lower bound with the upper bound, we have
\[d_{\HH}([-1,0], [1/3,\infty))\le d_{\HH}\left([-r^*,0], \left[g_K(x+3\eps)-g_K(x+\eps),\infty\right)\right).\]
This implies (\ref{eqn::extremal_aux}) and completes the proof.

\end{proof}
\begin{lemma}\label{lem::image_insideball}
Fix $z\in\overline{\HH}$ and $\eps>0$. Let $K$ be an $\HH$-hull and let $g_K$ be the conformal map from $\HH\setminus K$ onto $\HH$ normalized at $\infty$. Assume that
\[\dist(K, z)\ge 16\eps.\]
Then $g_K(B(z,\eps))$ is contained in the ball with center $g_K(z)$ and radius $4\eps |g_K'(z)|$. 
\end{lemma}
\begin{proof}
By Koebe 1/4 theorem, we know that 
\[\dist(g_K(K), g_K(z))\ge d:=4\eps |g_K'(z)|. \]
Let $h=g_K^{-1}$ restricted to $B(g_K(z), d)$. Applying Koebe 1/4 theorem to $h$, we know that 
\[\dist(z, \partial h(B(g_K(z), d)))\ge d|h'(g_K(z))|/4=\eps.\]
Therefore $h(B(g_K(z), d))$ contains the ball $B(z,\eps)$, and this implies that $B(g_K(z), d)$ contains the ball $g_K(B(z,\eps))$ as desired. 
\end{proof}

Loewner chain is a collection of $\HH$-hulls $(K_{t}, t\ge 0)$ associated with the family of conformal maps $(g_{t}, t\ge 0)$ obtained by solving the Loewner equation: for each $z\in\mathbb{H}$,
\begin{equation}\label{loewner}
\partial_{t}{g}_{t}(z)=\frac{2}{g_{t}(z)-W_{t}}, \quad g_{0}(z)=z,
\end{equation}
where $(W_t, t\ge 0)$ is a one-dimensional continuous function which we call the driving function. Let $T_z$ be the swallowing time of $z$ defined as $\sup\{t\ge 0: \min_{s\in[0,t]}|g_{s}(z)-W_{s}|>0\}$.
Let $K_{t}:=\overline{\{z\in\mathbb{H}: T_{z}\le t\}}$. Then $g_{t}$ is the unique conformal map from $H_{t}:=\mathbb{H}\backslash K_{t}$ onto $\mathbb{H}$ normalized at $\infty$.
\begin{lemma} \label{lem::tip_distance_large}
Suppose that $(K_t, t\ge 0)$ is a Loewner chain which is generated by a continuous curve $(\eta(t), t\ge 0)$. Fix $y\le -4r<0<x$. Let $\sigma$ be the first time that $\eta$ hits $B(y,r)$ and assume that $x$ is not swallowed by $\eta[0,\sigma]$ and that $y-r$ is not swallowed by $\eta[0,\sigma]$. 
Then we have 
\[g_{\sigma}(x)-W_{\sigma}\ge (x-y-2r)/2.\]
\end{lemma}
\begin{proof}
Let $\gamma$ be the right side of $\eta[0,\sigma]$.
We prove the conclusion by estimates on the extremal distance 
\[d_{\HH\setminus \eta[0,\sigma]}((-\infty, y-r), \gamma\cup[0,x]).\]
Denote $g_{\sigma}-W_{\sigma}$ by $f$. 
On the one hand, by the conformal invariance of the extremal distance, we have 
\begin{align*}
d_{\HH\setminus \eta[0,\sigma]}((-\infty, y-r), \gamma\cup[0,x])&=d_{\HH}((-\infty, f(y-r)), (0, f(x)))\\
&=d_{\HH}\left((-\infty, 0), \left(1, \frac{f(x)-f(y-r)}{-f(y-r)}\right)\right).
\end{align*}
On the other hand, by the comparison principle of the extremal distance \cite[Section 4.3]{AhlforsConformalInvariants}, we have 
\begin{align*}
d_{\HH\setminus \eta[0,\sigma]}((-\infty, y-r), \gamma\cup[0,x])&\le d_{\HH\setminus B(y,r)}((-\infty, y-r), (y+r, x))\\
&=d_{\HH}\left((-\infty, 0), \left(1, \frac{1}{2}+\frac{x-y}{4r}+\frac{r}{4(x-y)}\right)\right). 
\end{align*}
Comparing these two parts, we have 
\[\frac{f(x)-f(y-r)}{-f(y-r)}\ge \frac{1}{2}+\frac{x-y}{4r}+\frac{r}{4(x-y)}\ge \frac{1}{2}+\frac{x-y}{4r}.\]
Thus 
\[\frac{g_{\sigma}(x)-W_{\sigma}}{W_{\sigma}-g_{\sigma}(y-r)}\ge \frac{x-y}{4r}-\frac{1}{2}.\]
Combining with the following fact (since $g_t(x)-g_t(y-r)$ is increasing in $t$):
\[g_{\sigma}(x)-g_{\sigma}(y-r)\ge x-y+r,\]
we obtain 
\[g_{\sigma}(x)-W_{\sigma}\ge \frac{x-y-2r}{x-y+2r}(x-y+r)\ge (x-y-2r)/2.\]
This completes the proof. 
\end{proof}
Here we discuss a little about the evolution of a point $y\in\R$ under $g_t$. We assume $y\le 0$. There are two possibilities: if $y$ is not swallowed by $K_t$, then we define $Y_t=g_t(y)$; if $y$ is swallowed by $K_t$, then we define $Y_t$ to the be image of the leftmost of point of $K_t\cap\R$ under $g_t$. Suppose that $(K_t, t\ge 0)$ is generated by a continuous path $(\eta(t), t\ge 0)$ and that the Lebesgue measure of $\eta[0,\infty]\cap\R$ is zero. Then the process $Y_t$ is uniquely characterized by the following equation: 
\[Y_t=y+\int_0^t \frac{2ds}{Y_s-W_s},\quad Y_t\le W_t,\quad \forall t\ge 0.\] 
In this paper, we may write $g_t(y)$ for the process $Y_t$. 

\subsection{SLE processes} 
An $\SLE_{\kappa}$ is the random Loewner chain $(K_{t}, t\ge 0)$ driven by $W_t=\sqrt{\kappa}B_t$ where $(B_t, t\ge 0)$ is a standard one-dimensional Brownian motion.
In \cite{RohdeSchrammSLEBasicProperty}, the authors prove that $(K_{t}, t\ge 0)$ is almost surely generated by a continuous transient curve, i.e. there almost surely exists a continuous curve $\eta$ such that for each $t\ge 0$, $H_{t}$ is the unbounded connected component of $\mathbb{H}\backslash\eta[0,t]$ and that $\lim_{t\to\infty}|\eta(t)|=\infty$.

We can define an SLE$_{\kappa}(\underline{\rho}^L;\underline{\rho}^R)$ process with multiple force points $(\underline{x}^L; \underline{x}^R)$ where 
\[\underline{\rho}^L=(\rho^{l, L}, ..., \rho^{1, L}), \quad \underline{\rho}^R=(\rho^{1, R}, ..., \rho^{r, R}) \quad\text{ with }\rho^{i,q}\in \R;\]
\[\underline{x}^L=(x^{l, L}<\cdots<x^{1,L}\le 0),\quad \underline{x}^R=(0\le x^{1, R}<\cdots<x^{r, R}).\]
It is the Loewner chain driven by $W_{t}$ which is the solution to the following systems of SDEs:
\[dW_{t}=\sqrt{\kappa}dB_{t}+\sum_i \frac{\rho^{i, L} dt}{W_{t}-V^{i, L}_{t}}+\sum_i \frac{\rho^{i, R} dt}{W_{t}-V^{i, R}_{t}}, \quad W_{0}=0;\]
\[ dV^{i,L}_{t}=\frac{2dt}{V^{i,L}_{t}-W_{t}}, \quad V^{i, L}_{0}=x^{i, L};\quad 
dV^{i, R}_t=\frac{2dt}{V^{i, R}_t-W_t}, \quad V^{i, R}_0=x^{i,R}.\]
The solution exists and is unique up to the continuation threshold is hit---the first time $t$ that 
\[W_t=V_t^{j,q} \text{ where }\sum_1^j \rho^{i,q}\le -2,\text{ for some }q\in \{L, R\}.\]
Moreover, the corresponding Loewner chain is almost surely generated by a continuous curve (\cite[Section 2]{MillerSheffieldIG1}). 

In fact, in this paper, we only need the definitions and properties of $\SLE$ with three force points: $\SLE_{\kappa}(\rho^L;\rho^{1,R},\rho^{2,R})$ with force points $(x^L; x^{1,R}, x^{2,R})$. To simplify notations, we will focus on these SLE processes in this section. 
From Girsanov Theorem, it follows that the law of an $\SLE_{\kappa}(\underline{\rho})$ process can be constructed by reweighting the law of an ordinary $\SLE_{\kappa}$.
\begin{lemma}\label{lem::sle_mart}
Suppose $x^{L}<0<x^{1,R}<x^{2,R}$ and $\rho^L, \rho^{1,R}, \rho^{2,R}\in \R$. Define 
\begin{align*}
M_t&=g_t'(x^L)^{\rho^L(\rho^L+4-\kappa)/(4\kappa)}(g_t(x^L)-W_t)^{\rho^L/\kappa}\\
&\times g_t'(x^{1,R})^{\rho^{1,R}(\rho^{1,R}+4-\kappa)/(4\kappa)}(g_t(x^{1,R})-W_t)^{\rho^{1,R}/\kappa}\\
&\times g_t'(x^{2,R})^{\rho^{2,R}(\rho^{2,R}+4-\kappa)/(4\kappa)}(g_t(x^{2,R})-W_t)^{\rho^{2,R}/\kappa}\\
&\times (g_t(x^{1,R})-g_t(x^L))^{\rho^L\rho^{1,R}/(2\kappa)}(g_t(x^{2,R})-g_t(x^L))^{\rho^L\rho^{2,R}/(2\kappa)}
(g_t(x^{2,R})-g_t(x^{1,R}))^{\rho^{1,R}\rho^{2,R}/(2\kappa)}.
\end{align*}
Then $M$ is a local martingale for $\SLE_{\kappa}$ and the law of $\SLE_{\kappa}$ weighted by $M$ (up to the first time that $W$ hits one of the force points) is equal to the law of $\SLE_{\kappa}(\rho^{L};\rho^{1,R}, \rho^{2,R})$ with force points $(x^L; x^{1,R}, x^{2,R})$.  
\end{lemma}
\begin{proof}
\cite[Theorem 6]{SchrammWilsonSLECoordinatechanges}.
\end{proof}

Suppose $\eta$ is an $\SLE_{\kappa}(\rho^L;\rho^{1,R}, \rho^{2,R})$ process with force points $(x^L; x^{1,R}, x^{2,R})$. 
There are two special values of $\rho$: $\kappa/2-2$ and $\kappa/2-4$. 
If $\rho^{1,R}+\rho^{2,R}\ge \kappa/2-2$, then $\eta$ never hits $[x^{2,R},\infty)$. If $\rho^{1,R}+\rho^{2,R}\le \kappa/2-4$, then $\eta$ almost surely accumulates at $x^{2,R}$ at finite time. 
\begin{lemma}\label{lem::sle_positivechance} Fix $\kappa\in (0,4)$.  
Suppose that $\eta$ is an $\SLE_{\kappa}(\rho, \nu)$ process with force points $(v, x)$ where 
\[0\le v<x,\quad \rho>-2, \quad \rho+\nu< \kappa/2-4.\]
For $\eps>0$, define 
\[\tau=\inf\{t: \eta(t)\in B(x,\eps)\},\quad T=\inf\{t: \eta(t)\in [x,\infty)\}.\]
For $C\ge 4, 1/4\ge c>0$, define 
\[\LF=\{\tau<T, \eta[0,\tau]\subset B(0, Cx), \dist(\eta[0,\tau], [x-\eps, Cx])\ge c\eps\}.\]
Then, there exist constants $c, C, u_0>0$ which are uniform over $v,x, \eps$ such that $\PP[\LF]\ge u_0$.
\end{lemma}
\begin{proof}
By the scaling invariance of $\SLE$, we may assume $x=1$. Let $\varphi(z)=\eps z/(1-z)$. Then $\varphi$ is the M\"{o}bius transformation of $\HH$ that sends the triplet $(0,1,\infty)$ to $(0,\infty, -\eps)$. Denote the image of $\eta$ under $\varphi$ by $\tilde{\eta}$, and denote its law by $\tilde{\PP}$. Note that $\tilde{\eta}$ is an $\SLE_{\kappa}(\rho^L;\rho^R)$ with force points $(-\eps; \eps v/(1-v))$ where 
\[\rho^L=\kappa-6-\rho-\nu> \kappa/2-2,\quad \rho^R=\rho>-2.\]
For $r\in (0,1/4)$ and $y\in (-1,0)$, 
let $\tilde{T}=\inf\{t: \tilde{\eta}(t)\in\partial B(y, r|y|) \}$ and $\tilde{S}=
\inf\{t: \tilde{\eta}(t)\in \partial B(0,1)\}$. 
Since $\rho^L >\kappa/2-2$, by \cite[Corollary 3.3]{MillerWuSLEIntersection} or Lemma \ref{lem::onepoint_sle_multiple}, there exists $A>1$ depending only on $\kappa, \rho^L, \rho^R$ such that,  
\begin{equation}\label{eqn::MW_cor3.3}
\tilde{\PP}\left[\tilde{T}<\tilde{S}, \Im{\tilde{\eta}(\tilde{T})}\ge r|y|/4\right]\le r^A.
\end{equation}

 Consider the image of $\HH\setminus B(0, C)$ under $\varphi$. It is contained in the ball $B(-\eps, 2\eps/C)$. Since $\rho^L>\kappa/2-2$, there exists a function $q(C)$ such that the probability for $\tilde{\eta}$ to hit $B(-\eps, 2\eps/C)$ is bounded by $q(C)$ and $q(C)$ goes to zero as $C\to \infty$. Consider the image of $c\eps$-neighborhood of $[1+\eps, C]$ under $\varphi$. Since $c\eps$-neighborhood of $[1+\eps, C]$ is contained in the union of the balls $B(1+kc\eps/4, 4c\eps)$ for $4/c \le k\le C/\eps$, its image under $\varphi$ is contained in the union of the following balls
\[B(-4/(ck)-\eps, 256/(ck^2)),\quad \lfloor4/c\rfloor\le k\le \lceil C/\eps\rceil.\]
Define $\tilde{\LF}$ to be the event that $\tilde{\eta}$ exits the unit disc without touching the union of $B(-\eps, 2\eps/C)$ and the image of $c\eps$-neighborhood of $[1+\eps, C]$ under $\varphi$. Then, by (\ref{eqn::MW_cor3.3}), we have
\[1-\PP[\LF]\le 1-\tilde{\PP}[\tilde{\LF}] \lesssim q(C)+\sum_{k=4/c}^{C/\eps}\left(\frac{1}{4k+\eps ck^2}\right)^A\lesssim q(C)+c^{A-1}.\]
This implies the conclusion. 
\end{proof}

\begin{lemma}\label{lem::sle_positivechance2}
Fix $\kappa\in (0,4)$. Suppose that $\eta$ is an $\SLE_{\kappa}(\rho, \nu)$ process with force points $(v,x)$ where 
\[0\le v\le x,\quad \rho>-2, \quad \rho+\nu> \kappa/2-2.\] For $r>0>y$, assume $r<|y|\lesssim r$.  Let $\sigma$ be the first time that $\eta$ hits $B(y,r)$. 
For $C\ge 4, 1/4\ge c>0$, define 
\[\LF=\{\sigma<\infty, \dist(\eta[0,\sigma], x)\ge cx, \eta[0,\sigma]\subset B(0, C|y|), \dist(\eta[0,\sigma], [Cy, y])\ge cr\}.\]
Then, there exist constants $c, C, v_0>0$ which are uniform over $v,x, y$ such that $\PP[\LF]\ge v_0$.
\end{lemma}
\begin{proof}
Define 
\[\LG=\{\sigma<\infty, \eta[0,\sigma]\subset B(0, C|y|), \dist(\eta[0,\sigma], [Cy, y])\ge cr\}.\]
Since $r<|y|\lesssim r$, there exist constants $c, C, v_1>0$ which are uniform over $v,x,y$ such that $\PP[\LG]\ge v_1$. 

For $\delta>0$, consider the event $\{\dist(\eta, x)\ge \delta x\}$. By the scaling invariance, we know that the probability of this event only depends on $v/x$ and $\delta$, and we denote its probability by $f(v/x;\delta)$. We may assume $x=1$. By \cite[Section 4.7]{LawlerConformallyInvariantProcesses}, we know that $f(v;\delta)$ is continuous in $v$ and it is positive for any $v\in [0,1]$. Therefore, there is a function $f(\delta)>0$ such that $f(\delta)\to 0$ as $\delta\to 0$ and that $f(v; \delta)\ge 1-f(\delta)$. 
Therefore, $\PP[\LF]\ge v_1-f(\delta)$ where $f(\delta)\to 0$ as $\delta\to 0$. 
This implies the conclusion. 
\end{proof}

%% file: tex/sle_statements.tex
Fix $\kappa\in (0,4]$ and $\rho>-2, v>0$. Let $\eta$ be an $\SLE_{\kappa}(\rho)$ with force point $v$. Assume $y\le -4r<0<\eps\le v\le x$ and we consider the crossings of $\eta$ between $B(x,\eps)$ and $B(y,r)$. We have four different types of the crossing events. Let $T_x$ be the first time that $\eta$ swallows $x$. 

Set $\tau_0=\sigma_0=0$. Let $\tau_1$ be the first time that $\eta$ hits $B(x,\eps)$ and let $\sigma_1$ be the first time after $\tau_1$ that $\eta$ hits the connected component of $\partial B(y,r)\setminus \eta[0,\tau_1]$ containing $y-r$. For $j\ge 1$, let $\tau_j$ be the first time after $\sigma_{j-1}$ that $\eta$ hits the connected component of $\partial B(x,\eps)\setminus \eta[0,\sigma_{j-1}]$ containing $x+\eps$, and let $\sigma_j$ be the first time after $\tau_j$ that $\eta$ hits the connected component of $\partial B(y,r)\setminus \eta[0,\tau_j]$ containing $y-r$. Define 
\[\LH^{\alpha}_{2j-1}(\eps, x, y, r; v)=\{\tau_j<T_x\},\quad \LH^{\beta}_{2j}(\eps, x, y, r; v)=\{\sigma_j<T_x\}.\]
If $\rho\ge \kappa/2-2$, then these two events are the same; whereas when $\rho\in (-2, \kappa/2-2)$, these two events are distinct. 
In the definition of $\LH^{\alpha}_{2j-1}$ and $\LH^{\beta}_{2j}$, we are interested in the case when $x, y, r$ are fixed and $\eps>0$ small. Imagine that $\eta$ is the interface of the lattice model, then $\LH^{\alpha}_{2j-1}$ means that there are $2j-1$ arms connecting $B(x,\eps)$ to far away place; and $\LH^{\beta}_{2j}$ means that there are $2j$ arms connecting $B(x,\eps)$ to far away place. 

Next, we define the other two types of crossing events. Attention that we will change the definition of the stopping times. Set $\tau_0=\sigma_0=0$. Let $\sigma_1$ be the first time that $\eta$ hits $B(y,r)$ and $\tau_1$ be the first time after $\sigma_1$ that $\eta$ hits the connected component of $\partial B(x,\eps)\setminus \eta[0,\sigma_1]$ containing $x+\eps$. For $j\ge 1$, let $\sigma_j$ be the first time after $\tau_{j-1}$ that $\eta$ hits the connected component of $\partial B(y,r)\setminus \eta[0,\tau_{j-1}]$ containing $y-r$ and let $\tau_j$ be the first time after $\sigma_j$ that $\eta$ hits the connected component of $\partial B(x,\eps)\setminus \eta[0,\sigma_j]$ containing $x+\eps$. Define 
\[\LH^{\alpha}_{2j}(\eps, x, y, r; v)=\{\tau_j<T_x\},\quad \LH^{\beta}_{2j+1}(\eps, x, y, r; v)=\{\sigma_{j+1}<T_x\}.\]
If $\rho\ge \kappa/2-2$, then these two events are the same; whereas when $\rho\in (-2, \kappa/2-2)$, these two events are distinct. 
In the definition of $\LH^{\alpha}_{2j}$ and $\LH^{\beta}_{2j+1}$, we are interested in the case when $y, r$ are fixed and $x=\eps>0$ small. Imagine that $\eta$ is the interface of the lattice model, then $\LH^{\alpha}_{2j}$ means that there are $2j$ arms connecting $B(x,\eps)$ to far away place; and $\LH^{\beta}_{2j+1}$ means that there are $2j+1$ arms connecting $B(x,\eps)$ to far away place. 
The reason that we wish to change the definition of the stopping times will become clear during the proof. The definition here might be confusing at first sight, but these definitions avoid confusions in the proof. 

Propositions \ref{prop::sle_boundary_alpha} and \ref{prop::sle_boundary_beta} study the probability of $\LH^{\alpha}$ and $\LH^{\beta}$ when the force point $v$ is close to $x$; Proposition \ref{prop::sle_boundary_gamma} studies the probability of $\LH^{\alpha}$ and $\LH^{\beta}$ when the force point $v$ is far from $x$.  
\begin{proposition}\label{prop::sle_boundary_alpha}
Fix $\kappa\in (0,4)$ and $\rho\in (-2,0]$. Set $\alpha_0^+=0$. For $j\ge 1$, define
\[\alpha_{2j-1}^+=2j(2j+\rho+2-\kappa/2)/\kappa, \quad \alpha_{2j}^+=2j(2j+\rho+4-\kappa/2)/\kappa.\]
Suppose $r\ge 1\vee(200\eps)$. For $j\ge 1$, we have 
\begin{align}
\PP\left[\LH^{\alpha}_{2j-1}(\eps, x, y, r; v)\right]\lesssim x^{\alpha_{2j-2}^+-\alpha_{2j-1}^+}\eps^{\alpha^+_{2j-1}},\quad&\text{provided }0\le x-v\lesssim \eps,\text{ and } |y|\ge (40)^{2j-1}r, 
\label{eqn::sle_boundary_alpha_odd_upper}\\
\PP\left[\LH^{\alpha}_{2j}(\eps, x, y, r; v)\right]\lesssim x^{\alpha_{2j}^+-\alpha_{2j-1}^+}\eps^{\alpha^+_{2j-1}},\quad&\text{provided }0\le x-v\lesssim \eps,\text{ and } |y|\ge (40)^{2j}r,
\label{eqn::sle_boundary_alpha_even_upper}
\end{align}
where the constants in $\lesssim$ depend only on $\kappa, \rho, j$ and $r$. We also have 
\begin{align}
\PP\left[\LH^{\alpha}_{2j-1}(\eps, x, y, r; v)\right]\gtrsim x^{\alpha_{2j-2}^+-\alpha_{2j-1}^+}\eps^{\alpha^+_{2j-1}},\quad&\text{provided }0\le x-v\lesssim \eps, \text{ and }x\asymp r\le |y|\lesssim r, 
\label{eqn::sle_boundary_alpha_odd_lower}\\
\PP\left[\LH^{\alpha}_{2j}(\eps, x, y, r; v)\right]\gtrsim x^{\alpha_{2j}^+-\alpha_{2j-1}^+}\eps^{\alpha^+_{2j-1}},\quad&\text{provided }0\le x-v\lesssim \eps, \text{ and }r\le |y|\lesssim r,  
\label{eqn::sle_boundary_alpha_even_lower}
\end{align}
where the constants in $\gtrsim$ depend only on $\kappa, \rho, j$ and $r$. In particular, we have
\begin{align*}
\PP\left[\LH^{\alpha}_{2j-1}(\eps, x, y, r; v)\right]\asymp  \eps^{\alpha^+_{2j-1}},\quad &\text{provided }0\le x-v\lesssim \eps,\text{ and } x\asymp r\le (40)^{2j-1}r\le |y|\lesssim r;\\
\PP\left[\LH^{\alpha}_{2j}(\eps, x, y, r; v)\right]\asymp \eps^{\alpha^+_{2j}},\quad& \text{provided }x\asymp v\asymp \eps, \text{ and } (40)^{2j}r\le |y|\lesssim r.
\end{align*}
\end{proposition}
\begin{proposition}\label{prop::sle_boundary_beta}
Fix $\kappa\in (0,4)$ and $\rho\in (-2, \kappa/2-2)$. Set $\beta_0^+=0$. For $j\ge 1$, define
\[\beta_{2j-1}^+=2j(2j+\kappa/2-4-\rho)/\kappa, \quad \beta^+_{2j}=2j(2j+\kappa/2-2-\rho)/\kappa. \]
Suppose $r\ge 1\vee(200\eps)$. For $j\ge 1$, we have 
\begin{align}
\PP\left[\LH^{\beta}_{2j}(\eps, x, y, r; v)\right]\lesssim x^{\beta_{2j-1}^+-\beta_{2j}^+}\eps^{\beta^+_{2j}},\quad&\text{provided }0\le x-v\lesssim \eps,\text{ and } |y|\ge (40)^{2j}r, 
\label{eqn::sle_boundary_beta_even_upper}\\
\PP\left[\LH^{\beta}_{2j-1}(\eps, x, y, r; v)\right]\lesssim x^{\beta_{2j-1}^+-\beta_{2j-2}^+}\eps^{\beta^+_{2j-2}},\quad&\text{provided }0\le x-v\lesssim \eps,\text{ and } |y|\ge (40)^{2j-1}r,
\label{eqn::sle_boundary_beta_odd_upper}
\end{align}
where the constants in $\lesssim$ depend only on $\kappa, \rho, j$ and $r$. We also have 
\begin{align}
\PP\left[\LH^{\beta}_{2j}(\eps, x, y, r; v)\right]\gtrsim x^{\beta_{2j-1}^+-\beta_{2j}^+}\eps^{\beta^+_{2j}},\quad&\text{provided }0\le x-v\lesssim \eps, \text{ and }x\asymp r\le |y|\lesssim r, 
\label{eqn::sle_boundary_beta_even_lower}\\
\PP\left[\LH^{\beta}_{2j-1}(\eps, x, y, r; v)\right]\gtrsim x^{\beta_{2j-1}^+-\beta_{2j-2}^+}\eps^{\beta^+_{2j-2}},\quad&\text{provided }0\le x-v\lesssim \eps, \text{ and }r\le |y|\lesssim r,  
\label{eqn::sle_boundary_beta_odd_lower}
\end{align}
where the constants in $\gtrsim$ depend only on $\kappa, \rho, j$ and $r$. In particular, we have
\begin{align*}
\PP\left[\LH^{\beta}_{2j}(\eps, x, y, r; v)\right]\asymp  \eps^{\beta^+_{2j}},\quad &\text{provided }0\le x-v\lesssim \eps,\text{ and } x\asymp r\le (40)^{2j}r\le |y|\lesssim r;\\
\PP\left[\LH^{\beta}_{2j-1}(\eps, x, y, r; v)\right]\asymp \eps^{\beta^+_{2j-1}},\quad& \text{provided }x\asymp v\asymp \eps, \text{ and } (40)^{2j-1}r\le |y|\lesssim r,
\end{align*}
where the constants in $\asymp$ are uniform over $\eps$. 
\end{proposition}
\begin{proposition}\label{prop::sle_boundary_gamma}
Fix $\kappa\in (0,4)$ and $\rho\in (-2, \kappa/2-2)$. Set $\gamma_0^+=0$. For $j\ge 1$, define 
\[\gamma^+_{2j-1}=(2j+\rho)(2j+\rho+2-\kappa/2)/\kappa,\quad \gamma^+_{2j}=2j(2j+\kappa/2-2)/\kappa.\]
Define the event 
\[\LF=\{\tau_1<T_x, \eta[0,\tau_1]\subset B(0, Cx), \dist(\eta[0,\tau_1], [x-\eps, x+3\eps])\ge c\eps \},\]
where $c, C$ are the constant decided in Lemma \ref{lem::sle_positivechance}. 
For $j\ge 1$, we have 
\begin{align}
\PP\left[\LH^{\alpha}_{2j-1}(\eps, x, y, r; 0^+)\cap\LF\right]\asymp \eps^{\gamma^+_{2j-1}},\quad &\text{provided } Cx\le r\le (40)^{2j-1}r\le |y|\lesssim r,\label{eqn::sle_boundary_gamma_odd}\\
\PP\left[\LH^{\beta}_{2j}(\eps, x, y, r; 0^+)\cap\LF\right]\asymp \eps^{\gamma^+_{2j}},\quad &\text{provided } Cx\le r\le (40)^{2j}r\le |y|\lesssim r,\label{eqn::sle_boundary_gamma_even}
\end{align}
where the constants in $\asymp$ are uniform over $\eps$. 
\end{proposition}
The conclusions in Proposition \ref{prop::sle_boundary_gamma} are weaker than the ones in Propositions \ref{prop::sle_boundary_alpha} and \ref{prop::sle_boundary_beta}, but they are sufficient to derive the arm exponents for the critical Ising model. 

%% file: tex/sle_boundary_derivatives.tex
\begin{lemma}\label{lem::sle_boundary_derivative1}
Fix $\kappa\in (0,4)$ and $\rho>-2$, let $x\ge v> \eps>0$. Let $\eta$ be an $\SLE_{\kappa}(\rho)$ with force point $v$. Let $O_t$ be the image of the rightmost point of $\eta[0,t]\cap\R$ under $g_t$. 
Define 
\[\tau_{\eps}=\inf\{t: \eta(t)\in B(x,\eps)\}, \quad T=\inf\{t: \eta(t)\in [x,\infty)\}.\]
For $\lambda\ge 0$, define 
\[u_1(\lambda)=\frac{1}{\kappa}(\rho+4-\kappa/2)+\frac{1}{\kappa}\sqrt{4\kappa\lambda+(\rho+4-\kappa/2)^2}.\]
For $b\in \R$, assume 
\begin{equation}\label{eqn::sle_boundary_derivative1_assumption}
4b\ge (\lambda-b)(2\rho+\kappa(\lambda-b)+4-\kappa).
\end{equation}
If $x=v$, we have
\begin{equation}\label{eqn::sle_boundary_derivative1_upper}
\E\left[g_{\tau_{\eps}}'(x)^b (g_{\tau_{\eps}}(x)-W_{\tau_{\eps}})^{\lambda-b}\one_{\{\tau_{\eps}<T\}}\right]\asymp x^{-u_1(\lambda)}\eps^{u_1(\lambda)+\lambda-b},
\end{equation}
where the constants in $\asymp$ depend only on $\kappa, \rho, \lambda, b$.  
For $C\ge 4, 1/4\ge c>0$, define 
 \[\LF=\{\tau_{\eps}<T, \Im{\eta(\tau_{\eps})}\ge c\eps, \eta[0,\tau_{\eps}]\subset B(0, Cx), \dist(\eta[0,\tau_{\eps}], [-Cx, y+r])\ge cr\}.\]
There exist constants $C, c$ depending only on $\kappa$ and $\rho$ such that,   
for $0\le x-v\lesssim \eps$ and $x\asymp r\le |y|\lesssim r$, we have
\begin{equation}\label{eqn::sle_boundary_derivative1_lower}
\E\left[g_{\tau_{\eps}}'(x)^{\lambda}\one_{\LF}\right]\asymp x^{-u_1(\lambda)}\eps^{u_1(\lambda)},
\end{equation} 
where the constants in $\asymp$ depend only on $\kappa, \rho, \lambda$.
\end{lemma}
\begin{lemma}\label{lem::sle_boundary_integrable}
Fix $\kappa >0$ and $\nu\le \kappa/2-4$. Let $\eta$ be an $\SLE_{\kappa}(\nu)$ in $\HH$ from 0 to $\infty$ with force point $1$. Denote by $W$ the driving function, $V$ the evolution of the force point.  Let $O_t$ be the image of the rightmost point of $K_t\cap \R$ under $g_t$. Set $\Upsilon_t=(g_t(1)-O_t)/g_t'(1)$ and $\sigma(s)=\inf\{t: \Upsilon_t=e^{-2s}\}$.
Set $J_t=(V_t-O_t)/(V_t-W_t)$. Let $T_0=\inf\{t: \eta(t)\in [1,\infty)\}$. We have, for $\beta\ge 0$,
\begin{equation}\label{eqn::integrable_J_sigma}
\E\left[J_{\sigma(s)}^{-\beta}1_{\{\sigma(s)<T_0\}}\right]\asymp 1,\quad \text{when }8+2\nu+\kappa\beta<2\kappa,
\end{equation}
where the constants in $\asymp$ depend only on $\kappa, \nu, \beta$.
\end{lemma}
\begin{proof}
Since $0\le J_t\le 1$, we only need to show the upper bound. 
Define $X_t=V_t-W_t$.
We know that
\[dW_t=\sqrt{\kappa}dB_t+\frac{\nu dt}{W_t-V_t},\quad dV_t=\frac{2dt}{V_t-W_t},\]
where $B$ is a standard 1-dimensional Brownian motion.
By It\^o's formula, we have that
\[dJ_t=\frac{J_t}{X_t^2}\left(\kappa-\nu-2-\frac{2}{1-J_t}\right)dt+\frac{J_t}{X_t}\sqrt{\kappa}dB_t,\quad d\Upsilon_t=\Upsilon_t\frac{-2J_t dt}{X_t^2(1-J_t)}.\]
Recall that $\sigma(s)=\inf\{t: \Upsilon_t=e^{-2s}\}$, and denote by $\hat{X}, \hat{J}, \hat{\Upsilon}$ the processes indexed by $\sigma(s)$. Then we have that
\[d\sigma(s)=\hat{X}_s^2\frac{1-\hat{J}_s}{\hat{J}_s}ds,\quad d\hat{J}_s=\left(\kappa-\nu-4-(\kappa-\nu-2)\hat{J}_s\right)ds+\sqrt{\kappa \hat{J}_s(1-\hat{J}_s)}d\hat{B}_s,\]
where $\hat{B}$ is a standard 1-dimensional Brownian motion.
By \cite[Equations (56), (62)]{LawlerMinkowskiSLERealLine} and \cite[Appendix B]{ZhanErgodicityTipSLE}, we know that $\hat{J}$ has an invariant density on $(0,1)$, which is proportional to $y^{1-(8+2\nu)/\kappa}(1-y)^{4/\kappa-1}$.
Moreover, since $\hat J_0=1$, by a standard coupling argument, we may couple $(\hat J_s)$ with the stationary process $(\tilde J_s)$ that satisfies the same equation as $(\hat J_s)$, such that $\hat J_s\ge \tilde J_s$ for all $s\ge 0$. Then we get $\E[\hat{J}_s^{-\beta}]\le \E[\tilde J_s^{-\beta}]$, which is a finite constant if $8+2\nu+\kappa\beta<2\kappa$. This gives the upper bound in (\ref{eqn::integrable_J_sigma}) and completes the proof of (\ref{eqn::integrable_J_sigma}). 
\end{proof}

\begin{proof}[Proof of (\ref{eqn::sle_boundary_derivative1_upper})]
A similar estimate is derived in \cite[Proposition 4.15]{WuZhanSLEBoundaryArmExponents} for $\SLE_{\kappa}$. The proof is similar, to be self-contained, we give a complete proof here. 
Define $\Upsilon_t=(g_t(x)-O_t)/g_t'(x)$, $J_t=(g_t(x)-O_t)/(g_t(x)-W_t)$ and $\hat{\tau}_{\eps}=\inf\{t: \Upsilon_t=\eps\}$. 
Set
\[M_t=g_t'(x)^{(\nu-\rho)(\nu+\rho+4-\kappa)/(4\kappa)}(g_t(x)-W_t)^{(\nu-\rho)/\kappa},\]
where \[\nu=\kappa/2-4-\sqrt{4\kappa\lambda+(\rho+4-\kappa/2)^2}.\]
Then $M$ is a local martingale and the law of $\eta$ weighted by $M$ becomes the law of $\SLE_{\kappa}(\nu)$ with force point $x$. Set $\beta=u_1(\lambda)+\lambda-b$. By the choice of $\nu$, we can rewrite 
\[M_t=g_t'(x)^b(g_t(x)-W_t)^{\lambda-b}\Upsilon_t^{-\beta}J_t^{\beta}. \]
At time $\hat{\tau}_{\eps}<\infty$, we have $\Upsilon_{\hat{\tau}_{\eps}}=\eps$. Thus 
\[\E\left[g_{\hat{\tau}_{\eps}}'(x)^b (g_{\hat{\tau}_{\eps}}(x)-W_{\hat{\tau}_{\eps}})^{\lambda-b}\one_{\{\hat{\tau}_{\eps}<T_x\}}\right]= \eps^{\beta}M_0\E^*\left[\left(J^*_{\hat{\tau}_{\eps}^*}\right)^{-\beta}\right]=\eps^{\beta}x^{-u_1(\lambda)}\E^*\left[\left(J^*_{\hat{\tau}_{\eps}^*}\right)^{-\beta}\right]\asymp \eps^{\beta}x^{-u_1(\lambda)} ,\]
where $\PP^*$ is the law of $\eta$ weighted by $M$ and $\hat{\tau}_{\eps}^*, J^*$ are defined accordingly. The last relation is due to Lemma \ref{lem::sle_boundary_integrable}. Thus we have 
\begin{equation}\label{eqn::sle_boundary_derivative1_upper_conformalradius}
\E\left[g_{\hat{\tau}_{\eps}}'(x)^b (g_{\hat{\tau}_{\eps}}(x)-W_{\hat{\tau}_{\eps}})^{\lambda-b}\one_{\{\hat{\tau}_{\eps}<T_x\}}\right]\asymp x^{-u_1(\lambda)}\eps^{u_1(\lambda)+\lambda-b}.
\end{equation}

Consider the process $(U_t:=g_t'(x)^b(g_t(x)-W_t)^{\lambda-b})_{t\ge 0}$. We can check that it is a super martingale by It\^{o}'s formula when (\ref{eqn::sle_boundary_derivative1_assumption}) holds. Combining with the fact $\hat{\tau}_{\eps/4}\ge \tau_{\eps}\ge \hat{\tau}_{4\eps}$, we have
\[\E\left[U_{\hat{\tau}_{\eps/4}}\one_{\{\hat{\tau}_{\eps/4}<T_x\}}\right]\le \E\left[U_{\tau_{\eps}}\one_{\{\tau_{\eps}<T_x\}}\right]\le \E\left[U_{\hat{\tau}_{4\eps}}\one_{\{\hat{\tau}_{4\eps}<T_x\}}\right].\]
Combining with (\ref{eqn::sle_boundary_derivative1_upper_conformalradius}), we obtain (\ref{eqn::sle_boundary_derivative1_upper}).
\end{proof}

\begin{proof}[Proof of (\ref{eqn::sle_boundary_derivative1_lower})]
We may assume $x>v$. Define 
\[M_t=g_t'(x)^{\nu(\nu+4-\kappa)/(4\kappa)}(g_t(x)-W_t)^{\nu/\kappa}(g_t(x)-g_t(v))^{\nu\rho/(2\kappa)},\quad\text{where }\nu=-\kappa u_1(\lambda).\]
Then $M$ is a local martingale for $\eta$ and the law of $\eta$ weighted by $M$ is an $\SLE_{\kappa}(\rho, \nu)$ with force points $(v, x)$. 
We argue that 
\begin{equation}\label{eqn::sle_boundary_derivative_aux}
g_{\tau_{\eps}}(x)-g_{\tau_{\eps}}(v)\asymp (x-v)g_{\tau_{\eps}}'(x).
\end{equation}
There are two possibilities: $v$ is swallowed by $\eta[0,\tau_{\eps}]$ or not. If $v$ is not swallowed by $\eta[0,\tau_{\eps}]$, then by Koebe 1/4 theorem, we know that $g_{\tau_{\eps}}(x)-g_{\tau_{\eps}}(v)\asymp (x-v)g_{\tau_{\eps}}'(x)$. If $v$ is swallowed by $\eta[0,\tau_{\eps}]$, then we must have $x-v\ge \eps$. By Koebe 1/4 theorem, we have $g_{\tau_{\eps}}(x)-g_{\tau_{\eps}}(v)\asymp \eps g_{\tau_{\eps}}'(x)$. Since $\eps\le x-v\lesssim \eps$, we have $g_{\tau_{\eps}}(x)-g_{\tau_{\eps}}(v)\asymp (x-v)g_{\tau_{\eps}}'(x)$. These complete the proof of (\ref{eqn::sle_boundary_derivative_aux}).

On the event $\{\Im{\eta(\tau_{\eps})}\ge c\eps\}$, we also have $g_{\tau_{\eps}}(x)-W_{\tau_{\eps}}\asymp \eps g_{\tau_{\eps}}'(x)$. 
Combining with (\ref{eqn::sle_boundary_derivative_aux}) and the choice of $\nu$, we have 
\[M_{\tau_{\eps}}\asymp \eps^{\nu/\kappa}(x-v)^{\nu\rho/(2\kappa)}g_{\tau_{\eps}}'(x)^{\lambda},\quad \text{on }\LF.\]
Therefore,
\[\E\left[g_{\tau_{\eps}}'(x)^{\lambda}\one_{\LF}\right]\asymp \eps^{-\nu/\kappa}(x-v)^{-\nu\rho/(2\kappa)}M_0\PP^*[\LF^*]=\eps^{u_1(\lambda)}x^{-u_1(\lambda)}\PP^*[\LF^*],\]
where $\PP^*$ is the law of $\eta$ weighted by $M$ and $\LF^*$ is defined accordingly. Note that 
\[\rho>-2, \quad \rho+\nu<\kappa/2-4.\]
By a similar proof of Lemma \ref{lem::sle_positivechance}, we know that there are constants $C, c$ such that $\PP^*[\LF^*]\asymp 1$. This completes the proof. 
\end{proof}
\begin{remark}
In fact, Equation (\ref{eqn::sle_boundary_derivative1_upper_conformalradius}) is true for all $\kappa>0$ and $\rho\in\R$ as long as
\[\kappa\lambda-\kappa u_1(\lambda)+2\rho+8-2\kappa<\kappa b\le \kappa\lambda+\kappa u_1(\lambda).\]
\end{remark}
\begin{remark}\label{rem::sle_boundary_alpha_first}
Taking $\lambda=b=0$ in Lemma \ref{lem::sle_boundary_derivative1}, we know that  Proposition \ref{prop::sle_boundary_alpha} holds for $\LH^{\alpha}_1$ with  
\[\alpha^+_1=u_1(0)=2(\rho+4-\kappa/2)/\kappa.\]
Precisely, taking $\lambda=0$ in (\ref{eqn::sle_boundary_derivative1_lower}), we have 
\[\PP[\eta\text{ hits }B(x,\eps)]\gtrsim (\eps/x)^{u_1(0)}.\]
Taking $\lambda=b=0$ in (\ref{eqn::sle_boundary_derivative1_upper}), we have 
\[\PP[\eta\text{ hits }B(v,\eps)]\asymp (\eps/v)^{u_1(0)}.\]
Since $0\le x-v\lesssim \eps$, we know that $B(x, \eps)$ is contained in $B(v, \tilde{C}\eps)$ for some constant $\tilde{C}$, thus 
\[\PP[\eta\text{ hits }B(x,\eps)]\le \PP[\eta\text{ hits }B(v, \tilde{C}\eps)]\asymp (\eps/v)^{u_1(0)}\lesssim (\eps/x)^{u_1(0)}.\]
\end{remark}
\begin{lemma}\label{lem::sle_boundary_derivative2}
Fix $\kappa\in (0,4)$ and $\rho\in (-2,0]$, let $v>0$. Let $\eta$ be an $\SLE_{\kappa}(\rho)$ with force point $v$. For $r>0>y$, and $0<v\le x$, define 
\[\sigma=\inf\{t: \eta(t)\in B(y,r)\},\quad T=\inf\{t: \eta(t)\in [x,\infty)\}.\]
For $\lambda\ge 0$, define 
\[u_2(\lambda)=\frac{1}{\kappa}(\kappa/2-2-\rho)+\frac{1}{\kappa}\sqrt{4\kappa\lambda+(\kappa/2-2-\rho)^2}.\]
For $b\le u_2(\lambda)$ and $x\ge v$, we have
\begin{equation}\label{eqn::sle_boundary_derivative2_upper}
\E\left[g_{\sigma}'(x)^{\lambda}(g_{\sigma}(x)-W_{\sigma})^b\one_{\{\sigma<T\}}\right]\lesssim x^{u_2(\lambda)}(x-y-2r)^{b-u_2(\lambda)},
\end{equation}
where the constant in $\lesssim$ depends only on $\kappa, \rho, \lambda, b$. Assume $r<|y|\lesssim r$, define 
\[\LF=\{\sigma<T, \dist(\eta[0,\sigma], x)\ge cx, \eta[0,\sigma]\subset B(0, C|y|), \dist(\eta[0,\sigma], [Cy, y])\ge cr\},\]
where the constants $C,c$ are decided in Lemma \ref{lem::sle_positivechance2}. Then, for $b\le u_2(\lambda)$ and $x\ge v\ge (1-c)x$, we have
\begin{equation}\label{eqn::sle_boundary_derivative2_lower}
\E\left[g_{\sigma}'(x)^{\lambda}(g_{\sigma}(x)-W_{\sigma})^b\one_{\LF}\right]\gtrsim x^{u_2(\lambda)}|y|^{b-u_2(\lambda)},
\end{equation}
where the constant in $\gtrsim$ depends only on $\kappa, \rho, \lambda, b$. 
\end{lemma}
\begin{proof}[Proof of (\ref{eqn::sle_boundary_derivative2_upper})]
We may assume $x>v$. Set 
\[M_t=g_t'(x)^{\nu(\nu+4-\kappa)/(4\kappa)}(g_t(x)-W_t)^{\nu/\kappa}(g_t(x)-g_t(v))^{\nu\rho/(2\kappa)},\quad \text{where }\nu=\kappa u_2(\lambda)\ge 0.\]
By \cite[Theorem 6]{SchrammWilsonSLECoordinatechanges}, we know that $M$ is a local martingale for $\eta$. Note that $\nu\rho\le 0$ and that
\[g_t(x)-g_t(v)\le (x-v)g_t'(x).\]
Thus, 
\[M_t\ge g_t'(x)^{\lambda}(g_t(x)-W_t)^{u_2(\lambda)}(x-v)^{\nu\rho/(2\kappa)}.\]
Therefore,
\begin{align*}
\E\left[g_{\sigma}'(x)^{\lambda}(g_{\sigma}(x)-W_{\sigma})^b\one_{\{\sigma<T\}}\right]&\le (x-v)^{-\nu\rho/(2\kappa)} M_0\E^*\left[(g^*_{\sigma^*}(x)-W^*_{\sigma^*})^{b-u_2(\lambda)}\one_{\{\sigma^*<T^*\}}\right]\\
&=x^{u_2(\lambda)}\E^*\left[(g^*_{\sigma^*}(x)-W^*_{\sigma^*})^{b-u_2(\lambda)}\one_{\{\sigma^*<T^*\}}\right]\\
&\lesssim x^{u_2(\lambda)}(x-y-2r)^{b-u_2(\lambda)}\tag{by Lemma \ref{lem::tip_distance_large}},
\end{align*}
where $\PP^*$ is the law of $\eta$ weighted by $M$ and $g^*, W^*, \sigma^*, T^*$ are defined accordingly. This implies the conclusion.
\end{proof}
\begin{proof}[Proof of (\ref{eqn::sle_boundary_derivative2_lower})]
Assume the same notations as in the proof of (\ref{eqn::sle_boundary_derivative2_upper}). On the event $\{\dist(\eta[0,\sigma], x)\ge cx\}$, since $0\le x-v\le cx$, by Koebe 1/4 theorem, we have
\[g_t(x)-g_t(v)\ge (x-v)g_t'(x)/4.\]
Thus 
\[M_t\lesssim g_t'(x)^{\lambda}(g_t(x)-W_t)^{u_2(\lambda)}(x-v)^{\nu\rho/(2\kappa)}.\]
On the event $\{\eta[0,\sigma]\subset B(0,C|y|)\}$, we have \[g_{\sigma}(x)-W_{\sigma}\lesssim |y|.\]
Therefore, 
\[\E\left[g_{\sigma}'(x)^{\lambda}(g_{\sigma}(x)-W_{\sigma})^b\one_{\LF}\right]\gtrsim x^{u_2(\lambda)}|y|^{b-u_2(\lambda)}\PP^*[\LF^*],\]
where $\PP^*$ is the law of $\eta$ weighted by $M$ and $\LF^*$ is defined accordingly. By Lemma \ref{lem::sle_positivechance2}, we have $\PP^*[\LF^*]\asymp 1$. This completes the proof. 
\end{proof}
\begin{remark}\label{rem::sle_boundary_beta_first}
Taking $\lambda=b=0$ in Lemma \ref{lem::sle_boundary_derivative2}, we have 
\[\PP[\sigma<T_x]\asymp x^{u_2(0)},\quad \text{where }u_2(0)=2(\kappa/2-2-\rho)/\kappa. \]
This implies that Proposition \ref{prop::sle_boundary_beta} holds for $\LH^{\beta}_1$. 
\end{remark}

\begin{lemma} \label{lem::sle_boundary_derivative_gamma}
Fix $\kappa\in (0,4)$ and $\rho>-2$. Let $\eta$ be an $\SLE_{\kappa}(\rho)$ with force point $0^+$ and denote by $(V_t, t\ge 0)$ the evolution of the force point.
For $x>\eps>0$, define 
\[\tau=\inf\{t: \eta(t)\in B(x,\eps)\},\quad T=\inf\{t: \eta(t)\in [x,\infty)\}. \]
For $\lambda\ge 0$, define 
\[u_3(\lambda)=\frac{(\rho+2)}{2\kappa}\left(\rho+4-\kappa/2+\sqrt{4\kappa\lambda+(\rho+4-\kappa/2)^2}\right).\]
 Define 
\[\LG=\{\tau<T, \Im{\eta(\tau)}\ge c\eps\},\quad \LF=\LG\cap\{\eta[0,\tau]\subset B(0, Cx), \dist(\eta[0,\tau], [x-\eps, x+3\eps])\ge c\eps\}\]
where $c, C$ are the constants decided in Lemma \ref{lem::sle_positivechance}. Then we have
\[\E\left[g_{\tau}'(x)^{\lambda}\one_{\LF}\right]\asymp\E\left[g_{\tau}'(x)^{\lambda}\one_{\LG}\right]\asymp \eps^{u_3(\lambda)}x^{-u_3(\lambda)},\]
where the constants in $\asymp$ depend only on $\kappa, \rho, \lambda$. 
\end{lemma}
\begin{proof}
Set 
\[M_t=g_t'(x)^{\nu(\nu+4-\kappa)/(4\kappa)}(g_t(x)-W_t)^{\nu/\kappa}(g_t(x)-V_t)^{\nu\rho/(2\kappa)},\]
where 
\[\nu=\kappa/2-4-\rho-\sqrt{4\kappa\lambda+(\kappa/2-4-\rho)^2}.\]
Then $M$ is a local martingale and the law of $\eta$ weighted by $M$ becomes $\SLE_{\kappa}(\rho, \nu)$ with force points $(0^+, x)$. On the event $\LG$, we have 
\[g_{\tau}(x)-W_{\tau}\asymp g_{\tau}(x)-V_{\tau}\asymp \eps g_{\tau}'(x).\]
Combining with the choice of $\nu$, we have 
\[M_{\tau}\asymp g_{\tau}'(x)^{\lambda}\eps^{-u_3(\lambda)},\quad \text{on }\LG.\]
Therefore,
\[\E\left[g_{\tau}'(x)^{\lambda}\one_{\LG}\right]\asymp \eps^{u_3(\lambda)}M_0\PP^*[\LG^*]=\eps^{u_3(\lambda)}x^{-u_3(\lambda)}\PP^*[\LG^*],\quad \E\left[g_{\tau}'(x)^{\lambda}\one_{\LF}\right]\asymp\eps^{u_3(\lambda)}x^{-u_3(\lambda)}\PP^*[\LF^*],\]
where $\eta^*$ is an $\SLE_{\kappa}(\rho, \nu)$ with force points $(0^+, x)$, and $\PP^*$ denotes its law and $\LG^*, \LF^*$ are defined accordingly. By Lemma \ref{lem::sle_positivechance}, we have $\PP^*[\LF^*]\asymp 1$. This completes the proof. 
\end{proof}
\begin{remark}\label{rem::sle_boundary_gamma_first}
Taking $\lambda=0$ in Lemma \ref{lem::sle_boundary_derivative_gamma}, we have 
\[\PP[\LF]\asymp \eps^{u_3(0)}x^{-u_3(0)},\quad \text{where }u_3(0)=(\rho+2)(\rho+4-\kappa/2)/\kappa.\]
This implies that Proposition \ref{prop::sle_boundary_gamma} holds for $\LH^{\alpha}_1$. 
\end{remark}

%% file: tex/sle_boundary1.tex
\begin{lemma}\label{lem::sle_boundary_alpha_oddtoeven_upper}
For $j\ge 1$, assume (\ref{eqn::sle_boundary_alpha_odd_upper}) holds for $\LH^{\alpha}_{2j-1}$, then (\ref{eqn::sle_boundary_alpha_even_upper}) holds for $\LH^{\alpha}_{2j}$.
\end{lemma}
\begin{proof}
Let $\sigma$ be the first time that $\eta$ hits the ball $B(y, 16(40)^{2j-1}r)$. Denote $g_{\sigma}-W_{\sigma}$ by $f$. Let $\tilde{\eta}$ be the image of $\eta[\sigma, \infty)$ under $f$. We know that $\tilde{\eta}$ is an $\SLE_{\kappa}(\rho)$ with force point $f(v)$. Define $\tilde{\LH}^{\alpha}_{2j-1}$ for $\tilde{\eta}$. We have the following observations.
\begin{itemize}
\item Consider the image of $B(y,r)$ under $f$. By Lemma \ref{lem::image_insideball}, we know that $f(B(y,r))$ is contained in the ball with center $f(y)$ and radius $4rf'(y)$. By Koebe 1/4 theorem, we have 
\[|f(y)|\ge 4(40)^{2j-1}rf'(y).\]
\item Consider the image of the connected component of $\partial B(x,\eps)\setminus \eta[0,\sigma]$ containing $x+\eps$ under $f$. By Lemma \ref{lem::extremallength_argument}, we know that it is contained in the ball with center $f(x+3\eps)$ and radius $8\eps f'(x+3\eps)$. Moreover, we have
\[f(x+3\eps)-f(v)\le (x+3\eps-v)f'(x+3\eps)\lesssim \eps f'(x+3\eps). \]
\end{itemize}
Combining these two facts with (\ref{eqn::sle_boundary_alpha_odd_upper}), we have 
\begin{align*}
\PP\left[\LH^{\alpha}_{2j}(\eps, x, y, r; v)\cond \eta[0,\sigma]\right]&\le \PP\left[\tilde{\LH}^{\alpha}_{2j-1}(8\eps f'(x+3\eps), f(x+3\eps), f(y), 4rf'(y); f(v))\right]\\
&\lesssim \left(g_{\sigma}(x+3\eps)-W_{\sigma}\right)^{\alpha^+_{2j-2}-\alpha^+_{2j-1}}\left(\eps g_{\sigma}'(x+3\eps)\right)^{\alpha^+_{2j-1}}.
\end{align*}
By Lemma \ref{lem::sle_boundary_derivative2} and the fact that the swallowing time of $x+3\eps$ is greater than $T_x$, we have 
\begin{align*}
\PP\left[\LH^{\alpha}_{2j}(\eps, x, y, r; v)\right]&\lesssim \E\left[\left(g_{\sigma}(x+3\eps)-W_{\sigma}\right)^{\alpha^+_{2j-2}-\alpha^+_{2j-1}}\left(\eps g_{\sigma}'(x+3\eps)\right)^{\alpha^+_{2j-1}}\one_{\{\sigma<T_x\}}\right]\\
&\lesssim \eps ^{\alpha_{2j-1}^+}(x+3\eps)^{u_2(\alpha^+_{2j-1})}(x-y-32(40)^{2j-1}r)^{\alpha_{2j-2}^+-\alpha_{2j}^+}\\
&\lesssim  x^{\alpha_{2j}^+-\alpha^+_{2j-1}}\eps^{\alpha^+_{2j-1}}. 
\end{align*}
The last line is because $x\ge \eps$ and $|y|\ge (40)^{2j}r$. 
\end{proof}
\begin{lemma}\label{lem::sle_boundary_alpha_eventoodd_upper}
For $j\ge 1$, assume (\ref{eqn::sle_boundary_alpha_even_upper}) holds for $\LH^{\alpha}_{2j}$, then (\ref{eqn::sle_boundary_alpha_odd_upper}) holds for $\LH^{\alpha}_{2j+1}$. 
\end{lemma}
\begin{proof}
If $x\le 64\eps$, then 
\[\PP\left[\LH^{\alpha}_{2j+1}(\eps, x, y, r; v)\right]\le \PP\left[\LH^{\alpha}_{2j}(\eps, x, y, r; v)\right].\]
This gives the conclusion. In the following, 
we may assume $x>64\eps$. Let $\tau$ be the first time that $\eta$ hits $B(x,16\eps)$. 
Denote $g_{\tau}-W_{\tau}$ by $f$. Let $\tilde{\eta}$ be the image of $\eta[\tau,\infty)$ under $f$. We know that $\tilde{\eta}$ is an $\SLE_{\kappa}(\rho)$ with force point $f(v)$. Define $\tilde{\LH}^{\alpha}_{2j}$ for $\tilde{\eta}$. We have the following observations.
\begin{itemize}
\item Consider the image of the connected component of $\partial B(y,r)\setminus \eta[0,\tau]$ containing $y-r$ under $f$. By Lemma \ref{lem::extremallength_argument}, we know that it is contained in the ball with center $f(y-3r)$ and radius $8rf'(y-3r)$. 
By Lemma \ref{lem::tip_distance_large}, we have 
\[|f(y-3r)|\ge (x-y+3r-32\eps)/2\ge (40)^{2j} 8r.\]
\item Consider the image of $B(x,\eps)$ under $f$. By Lemma \ref{lem::image_insideball}, we know that $B(x,\eps)$ is contained in the ball with center $f(x)$ and radius $4\eps f'(x)$. Moreover, 
\[f(x)-f(v)\le (x-v)f'(x)\lesssim \eps f'(x).\]
\end{itemize}
Combining these two facts with (\ref{eqn::sle_boundary_alpha_even_upper}), we have 
\begin{align*}
\PP\left[\LH^{\alpha}_{2j+1}(\eps, x, y, r; v)\cond \eta[0,\tau]\right]&\le \PP\left[\tilde{\LH}^{\alpha}_{2j}(4\eps f'(x), f(x), f(y-3r), 8rf'(y-3r); f(v))\right]\\
&\lesssim (g_{\tau}(x)-W_{\tau})^{\alpha_{2j}^+-\alpha^+_{2j-1}}\left(\eps g_{\tau}'(x)\right)^{\alpha_{2j-1}^+}.
\end{align*}
If $x=v$, by Lemma \ref{lem::sle_boundary_derivative1}, since $\alpha_{2j-1}^+$ and $\alpha_{2j}^+$ satisfy (\ref{eqn::sle_boundary_derivative1_assumption}):
\[\kappa\left(\alpha_{2j}^+-\alpha_{2j-1}^+\right)\left(2\rho+4-\kappa+\kappa\left(\alpha_{2j}^+-\alpha_{2j-1}^+\right)\right)=4j(2\rho+4-\kappa+4j)=4\kappa\alpha_{2j-1}^+,\]
we have 
\begin{align*}
\PP\left[\LH^{\alpha}_{2j+1}(\eps, v, y, r; v)\right]&\lesssim \E\left[(g_{\tau}(v)-W_{\tau})^{\alpha_{2j}^+-\alpha^+_{2j-1}}\left(\eps g_{\tau}'(v)\right)^{\alpha_{2j-1}^+}\one_{\{\hat{\tau}<T_v\}}\right]\\
&\lesssim v^{-u_1(\alpha^+_{2j})}\eps^{\alpha_{2j+1}^+}=v^{\alpha^+_{2j}-\alpha^+_{2j+1}}\eps^{\alpha^+_{2j+1}}.
\end{align*}
For $0\le x-v\lesssim \eps$, we know that $B(x,\eps)$ is contained in $B(v, \tilde{C}\eps)$ for some constant $\tilde{C}$, thus 
\[\PP\left[\LH^{\alpha}_{2j+1}(\eps, x, y, r; v)\right]\le \PP\left[\LH^{\alpha}_{2j+1}(\tilde{C}\eps, v, y, r; v)\right]\lesssim v^{\alpha^+_{2j}-\alpha^+_{2j+1}}\eps^{\alpha^+_{2j+1}}\lesssim x^{\alpha^+_{2j}-\alpha^+_{2j+1}}\eps^{\alpha^+_{2j+1}}.\]
This gives the conclusion. 
\end{proof}

\begin{lemma}\label{lem::sle_boundary_alpha_oddtoeven_lower}
For $j\ge 1$, assume (\ref{eqn::sle_boundary_alpha_odd_lower}) holds for $\LH^{\alpha}_{2j-1}$, then (\ref{eqn::sle_boundary_alpha_even_lower}) holds for $\LH^{\alpha}_{2j}$. 
\end{lemma}
\begin{proof}
Let $\sigma$ be the first time that $\eta$ hits $B(y,r)$. Define 
\[\LF=\{\sigma<T_x, \dist(\eta[0,\sigma], x)\ge cx, \eta[0,\sigma]\subset B(0,C|y|), \dist(\eta[0,\sigma], [Cy, y])\ge cr\},\]
where $c, C$ are the constants decided in Lemma \ref{lem::sle_positivechance}. Denote $g_{\sigma}-W_{\sigma}$ by $f$. Let $\tilde{\eta}$ be the image of $\eta[\sigma, \infty)$ under $f$, then $\tilde{\eta}$ is an $\SLE_{\kappa}(\rho)$ with force point $f(v)$. 
Given $\eta[0,\sigma]$ and on the event $\LF$, we have the following observations.
\begin{itemize}
\item Consider the image of $B(y,r)$ under $f$. By Koebe 1/4 theorem, it contains the ball with center $f(y)$ and radius $rf'(y)/4$. On the event $\{\dist(\eta[0,\sigma], [Cy, y])\ge cr\}$, we have 
\[rf'(y)/4\le |f(y)|\lesssim rf'(y).\]
\item Consider the image of $B(x,\eps)$ under $f$. On the event $\{\dist(\eta[0,\sigma], x)\ge cx\}$, by Koebe 1/4 theorem, it contains the ball with the center $f(x)$ and radius $c\eps f'(x)/4$. Since $x-v\lesssim \eps$, we have
\[f(x)-f(v)\le (x-v)f'(x)\lesssim \eps f'(x).\]
\item Compare $f(x)$ and $|f(y)|\asymp rf'(y)$. On the event $\{\eta[0,\sigma]\subset B(0, C|y|)\}$, we have $f(x)\lesssim |y|$. On the event $\{\dist(\eta[0,\sigma], [Cy, y])\ge cr\}$, we have $|f(y)|\gtrsim |y|$. Thus, on $\LF$, we have 
\[f(x)\lesssim |y|\lesssim |f(y)|\asymp rf'(y).\]
\end{itemize}
Combining these three facts with (\ref{eqn::sle_boundary_alpha_odd_lower}), we have 
\begin{align*}
\PP\left[\LH^{\alpha}_{2j}(\eps, x, y, r; v)\cond \eta[0,\sigma], \LF\right]&\ge\PP\left[\tilde{\LH}^{\alpha}_{2j-1}(\eps f'(x)/4, f(x), f(y), r f'(y)/4; f(v))\right]\\
&\gtrsim (g_{\sigma}(x)-W_{\sigma})^{\alpha^+_{2j-2}-\alpha^+_{2j-1}}\left(\eps g_{\sigma}'(x)\right)^{\alpha^+_{2j-1}}. 
\end{align*}
By Lemma \ref{lem::sle_boundary_derivative2}, we have 
\begin{align*}
\PP\left[\LH^{\alpha}_{2j}(\eps, x, y, r; v)\cap\LF\right]&\gtrsim \E\left[(g_{\sigma}(x)-W_{\sigma})^{\alpha^+_{2j-2}-\alpha^+_{2j-1}}\left(\eps g_{\sigma}'(x)\right)^{\alpha^+_{2j-1}}\one_{\LF}\right]\\
&\gtrsim x^{u_2(\alpha^+_{2j-1})}\eps^{\alpha_{2j-1}^+}=x^{\alpha^+_{2j}-\alpha^+_{2j-1}}\eps^{\alpha_{2j-1}^+}.
\end{align*}
This gives the conclusion. 
\end{proof}

\begin{lemma}\label{lem::sle_boundary_alpha_eventoodd_lower}
For $j\ge 1$, assume (\ref{eqn::sle_boundary_alpha_even_lower}) holds for $\LH^{\alpha}_{2j}$, then (\ref{eqn::sle_boundary_alpha_odd_lower}) holds for $\LH^{\alpha}_{2j+1}$. 
\end{lemma}
\begin{proof}
Let $\tau$ be the first time that $\eta$ hits $B(x,\eps)$. Define 
\[\LF=\{\tau<T_x, \Im{\eta(\tau)}\ge c\eps, \eta[0,\tau]\subset B(0, Cx), \dist(\eta[0,\tau], [-Cx, y+r])\ge cr\},\]
where $c, C$ are constants decided in Lemma \ref{lem::sle_boundary_derivative1}.  
Denote $g_{\tau}-W_{\tau}$ by $f$. Let $\tilde{\eta}$ be the image of $\eta[\tau, \infty)$ under $f$, then $\tilde{\eta}$ is an $\SLE_{\kappa}(\rho)$ with force point $f(v)$. Define $\tilde{\LH}^{\alpha}_{2j}$ for $\tilde{\eta}$. 
Given $\eta[0,\tau]$ and on the event $\LF$, we have the following observations.
\begin{itemize}
\item Consider the image of $B(y,r)$ under $f$. On the event $\LF$, we know that $f(B(y,r))$ contains the ball with center $f(y)$ and radius $crf'(y)/4$; moreover, we have 
\[crf'(y)/4\le |f(y)|\lesssim rf'(y).\]
\item Consider the image of $B(x,\eps)$ under $f$. By Koebe 1/4 theorem, it contains the ball with center $f(x)$ and radius $\eps f'(x)/4$. On the event $\{\Im{\eta(\tau)}\ge c\eps\}$, we have 
\[f(x)\asymp \eps f'(x). \]
 Since $x-v\lesssim\eps$, we have 
\[f(x)-f(v)\le (x-v)f'(x)\lesssim \eps f'(x).\]
\end{itemize}
Combining these two facts with (\ref{eqn::sle_boundary_alpha_even_lower}), we have 
\[\PP\left[\LH^{\alpha}_{2j+1}(\eps, x, y, r; v)\cond \eta[0,\tau], \LF\right]\ge \PP\left[\tilde{\LH}^{\alpha}_{2j}(\eps f'(x)/4, f(x), f(y), rf'(y)/4; f(v))\right]
\gtrsim (\eps g'_{\tau}(x))^{\alpha^+_{2j}}.\]
By Lemma \ref{lem::sle_boundary_derivative1}, we have 
\[\PP\left[\LH^{\alpha}_{2j+1}(\eps, x, y, r; v)\cap \LF\right]\gtrsim \E\left[ (\eps g'_{\tau}(x))^{\alpha^+_{2j}}\one_{\LF}\right]\asymp x^{-u_1(\alpha^+_{2j})}\eps^{u_1(\alpha^+_{2j})+\alpha^+_{2j}}=x^{\alpha^+_{2j}-\alpha^+_{2j+1}}\eps^{\alpha^+_{2j+1}}. \]
\end{proof}

\begin{proof}[Proof of Proposition \ref{prop::sle_boundary_alpha}]
Combining Remark \ref{rem::sle_boundary_alpha_first} with Lemmas \ref{lem::sle_boundary_alpha_oddtoeven_upper},
\ref{lem::sle_boundary_alpha_eventoodd_upper}, \ref{lem::sle_boundary_alpha_oddtoeven_lower} and  \ref{lem::sle_boundary_alpha_eventoodd_lower}, we obtain the conclusion. Note that 
\[\alpha^+_{2j+1}=\alpha^+_{2j}+u_1(\alpha^+_{2j}),\quad \alpha^+_{2j}=\alpha^+_{2j-1}+u_2(\alpha^+_{2j-1}).\]
\end{proof}
\begin{proof}[Proof of Proposition \ref{prop::sle_boundary_beta}]
By Remark \ref{rem::sle_boundary_beta_first}, we know the conclusion is true for $\LH^{\beta}_1$. Note that 
\[\beta^+_{2j}=\beta^+_{2j-1}+u_1(\beta^+_{2j-1}),\quad \beta^+_{2j+1}=\beta^+_{2j}+u_2(\beta^+_{2j}).\]
Moreover, the exponents $\beta_{2j-2}^+$ and $\beta_{2j-1}^+$ satisfy (\ref{eqn::sle_boundary_derivative1_assumption}): 
\[\kappa\left(\beta_{2j-1}^+-\beta_{2j-2}^+\right)\left(2\rho+4-\kappa+\kappa\left(\beta_{2j-1}^+-\beta_{2j-2}^+\right)\right)=8(j-1)(2j+\kappa/2-4-\rho)=4\kappa\beta_{2j-2}^+.\]

\noindent By the same proof of Lemma \ref{lem::sle_boundary_alpha_eventoodd_upper}, we have that, if (\ref{eqn::sle_boundary_beta_odd_upper}) holds for $\LH^{\beta}_{2j-1}$, then (\ref{eqn::sle_boundary_beta_even_upper}) holds for $\LH^{\beta}_{2j}$. 

\noindent By the same proof of Lemma \ref{lem::sle_boundary_alpha_oddtoeven_upper}, we have that, if (\ref{eqn::sle_boundary_beta_even_upper}) holds for $\LH^{\beta}_{2j}$, then (\ref{eqn::sle_boundary_beta_odd_upper}) holds for $\LH^{\beta}_{2j+1}$. 

\noindent By the same proof of Lemma \ref{lem::sle_boundary_alpha_eventoodd_lower}, we have that, if (\ref{eqn::sle_boundary_beta_odd_lower}) holds for $\LH^{\beta}_{2j-1}$, then (\ref{eqn::sle_boundary_beta_even_lower}) holds for $\LH^{\beta}_{2j}$. 

\noindent By the same proof of Lemma \ref{lem::sle_boundary_alpha_oddtoeven_lower}, we have that, if (\ref{eqn::sle_boundary_beta_even_lower}) holds for $\LH^{\beta}_{2j}$, then (\ref{eqn::sle_boundary_beta_odd_lower}) holds for $\LH^{\beta}_{2j+1}$.

\noindent Combining all these, we complete the proof. 
\end{proof}

%% file: tex/sle_boundary2.tex
\begin{proof}[Proof of (\ref{eqn::sle_boundary_gamma_odd}), Upper Bound] By Remark \ref{rem::sle_boundary_gamma_first}, we know that the conclusion is true for $\LH^{\alpha}_1$. 
We will prove the conclusion for $\LH^{\alpha}_{2j+1}$ for $j\ge 1$. 
Recall that $\eta$ is an $\SLE_{\kappa}(\rho)$ with force point $0^+$. Let $\tau$ be the first time that $\eta$ hits $B(x,\eps)$, and $T$ be the first time that $\eta$ swallows $x$. Recall that 
\[\LF=\{\tau<T, \eta[0,\tau]\subset B(0, Cx), \dist(\eta[0,\tau], [x-\eps, x+3\eps])\ge c\eps\}.\]
Given $\eta[0,\tau]$, denote $g_{\tau}-W_{\tau}$ by $f$. Let $\tilde{\eta}$ be the image of $\eta[\tau, \infty)$ under $f$, then $\tilde{\eta}$ is an $\SLE_{\kappa}(\rho)$ with force point $f(0^+)$. Define $\tilde{\LH}^{\alpha}_{2j}$ for $\tilde{\eta}$. 
We have the following observations.
\begin{itemize}
\item Consider the image of the connected component of $\partial B(x,\eps)\setminus \eta[0,\tau]$ containing $x+\eps$ under $f$. By Lemma \ref{lem::extremallength_argument}, we know that it is contained in the ball with center $f(x+3\eps)$ and radius $8\eps f'(x+3\eps)$. On the event $\{\dist(\eta[0,\tau], [x-\eps, x+3\eps])\ge c\eps\}$, by Koebe distorsion theorem \cite[Chapter I, Theorem 1.3]{Pommerenke}, we know that there exists some universal constant $\tilde{C}$ such that the ball with center $f(x+3\eps)$ and radius $8\eps f'(x+3\eps)$ is contained in the ball with center $f(x)$ and radius $\tilde{C}\eps f'(x)$. Moreover, on the event $\{\dist(\eta[0,\tau], [x-\eps, x+3\eps])\ge c\eps\}$, we have 
\[f(x)\asymp f(x)-f(0^+)\asymp \eps f'(x).\]
\item Consider the image of the connected component of $\partial B(y,r)\setminus \eta[0,\tau]$ containing $y-r$ under $f$. By Lemma \ref{lem::extremallength_argument}, we know that it is contained in the ball with center $f(y-3r)$ and radius $8rf'(y-3r)$. By Lemma \ref{lem::tip_distance_large}, we know that 
\[|f(y-3r)|\ge (x-y+3r-2\eps)/2\ge |y|/2\ge (40)^{2j} 8r.\] 
\end{itemize} 
Combining these two facts with (\ref{eqn::sle_boundary_alpha_even_upper}), we have 
\[\PP\left[\LH^{\alpha}_{2j+1}(\eps, x, y, r; 0^+)\cond \eta[0,\tau], \LF\right]\lesssim (\eps g_{\tau}'(x))^{\alpha^+_{2j}}.\]
By Lemma \ref{lem::sle_boundary_derivative_gamma}, we have 
\[\PP\left[\LH^{\alpha}_{2j+1}(\eps, x, y, r; 0^+)\cap\LF\right]\lesssim\E\left[ (\eps g_{\tau}'(x))^{\alpha^+_{2j}}\one_{\LF}\right]\asymp \eps^{u_3(\alpha^+_{2j})+\alpha^+_{2j}}.\]
Note that 
\[\gamma^+_{2j+1}=u_3(\alpha^+_{2j})+\alpha^+_{2j}.\]
This completes the proof. 
\end{proof}
\begin{proof}[Proof of (\ref{eqn::sle_boundary_gamma_odd}), Lower Bound]
Assume the same notations as in the proof of the upper bound. We have the following observations.
\begin{itemize}
\item Consider the image of $B(x,\eps)$ under $f$. By Koebe 1/4 theorem, it contains the ball with center $f(x)$ and radius $\eps f'(x)/4$. Moreover, on the event $\LF$, we have 
\[f(x)\asymp f(x)-f(0^+)\asymp \eps f'(x).\]
\item Consider the image of $B(y,r)$ under $f$.  Note that $r\ge Cx$ and $|y|\ge (40)^{2j+1}r$. Thus, on the event $\{\eta[0,\tau]\subset B(0, Cx)\}$, we know that $\eta[0,\tau]$ does not hit $B(y,r)$. Thus $f(B(y,r))$ contains the ball with center $f(y)$ and radius $rf'(y)/4$. On the event $\{\eta[0,\tau]\subset B(0, Cx)\}$, we know that 
\[rf'(y)/4\le |f(y)|\le |y|+ (Cx)^2/|y|\le 2|y|\asymp r. \] 
\end{itemize}
Combining these two facts with (\ref{eqn::sle_boundary_alpha_even_lower}), we have 
\[\PP\left[\LH^{\alpha}_{2j+1}(\eps, x, y, r; 0^+)\cond \eta[0,\tau], \LF\right]\gtrsim (\eps g_{\tau}'(x))^{\alpha^+_{2j}}.\]
By Lemma \ref{lem::sle_boundary_derivative_gamma}, we have 
\[\PP\left[\LH^{\alpha}_{2j+1}(\eps, x, y, r; 0^+)\cap\LF\right]\gtrsim\E\left[ (\eps g_{\tau}'(x))^{\alpha^+_{2j}}\one_{\LF}\right]\asymp \eps^{u_3(\alpha^+_{2j})+\alpha^+_{2j}}.\]
This completes the proof. 
\end{proof}
\begin{proof}[Proof of (\ref{eqn::sle_boundary_gamma_even})]
By the same proof of (\ref{eqn::sle_boundary_gamma_odd}), we could prove that 
\[\PP\left[\LH^{\alpha}_{2j}(\eps, x, y, r; 0^+)\cap \LF\right]\asymp \E\left[(\eps g'_{\tau}(x))^{\beta^+_{2j-1}}\right]\asymp \eps^{u_3(\beta^+_{2j-1})+\beta^+_{2j-1}}.\]
Note that 
\[\gamma^+_{2j}=u_3(\beta^+_{2j-1})+\beta^+_{2j-1}.\]
This completes the proof. 
\end{proof}

%% file: tex/sle_interior.tex
Fix $\kappa\in (0,4)$ and let $\eta$ be an $\SLE_{\kappa}$ in $\HH$ from 0 to $\infty$. 
Fix $z\in \HH$ with $|z|=1$ and suppose $r>0$ and $ y\le -4r$.
 Let $\tau_1$ be the first time that $\eta$ hits $B(z,\eps)$. Define 
\[\LE_2(\eps, z)=\{\tau_1<\infty\}.\]
Let $\sigma_1$ be the first time after $\tau_1$ that $\eta$ hits the connected component of $\partial B(y,r)\setminus \eta[0,\tau_1]$ containing $y-r$. Define $\LE^g$ to be the event that $z$ is in the unbounded connected component of $\HH\setminus (\eta[0,\sigma_1]\cup B(y,r))$.

Given $\eta[0,\sigma_1]$, we know that $B(z,\eps)\setminus \eta[0,\sigma_1]$ has one connected component that contains $z$, denoted by $C_z$. 
The boundary $\partial C_z$ consists of pieces of $\eta[0,\sigma_1]$ and pieces of $\partial B(z,\eps)$. Consider $\partial C_z\cap \partial B(z,\eps)$, there may be several connected components, but there is only one which can be connected to $\infty$ in $\HH\setminus (\eta[0,\sigma_1]\cup B(z,\eps))$. We denote this connected component by $C_z^b$, oriented it clockwise and denote the end point as $X_z^b$. See Figure \ref{fig::sle_interior_def}. 

Let $\tau_2$ be the first time after $\sigma_1$ that $\eta$ hits $C_z^b$, and let $\sigma_2$ be the first time after $\tau_2$ that $\eta$ hits the connected component of $\partial B(y,r)\setminus \eta[0,\tau_2]$ containing $y-r$. For $j\ge 2$, let $\tau_j$ be the first time after $\sigma_{j-1}$ such that $\eta$ hits the connected component of $C_z^b\setminus \eta[0,\sigma_{j-1}]$ containing $X_z^b$ and let $\sigma_j$ be the first time after $\tau_j$ that $\eta$ hits the connected component of $\partial B(y,r)\setminus \eta[0,\tau_j]$ containing $y-r$. For $j\ge 2$, define 
\[\LE_{2j}(\eps, z, y, r)=\LE^g\cap \{\tau_j<T_z\}.\] We will prove the following estimate on the probability of $\LE_{2j}$.
\begin{proposition}
Fix $\kappa\in (0,4)$ and $z\in\HH$ with $|z|=1$. For $j\ge 1$, define 
\[\alpha_{2j}=(16j^2-(\kappa-4)^2)/(8\kappa).\]
Define 
\[\LF=\{\eta[0,\tau_1]\subset B(0,R)\},\]
where $R$ is a constant decided in Lemma \ref{lem::sle_interior_derivative}. Then we have, for $j\ge 1$,
\begin{equation}\label{eqn::sle_interior}
\PP\left[\LE_{2j}(\eps, z, y, r)\cap \LF\right]= \eps^{\alpha_{2j}+o(1)},\quad \text{provided } R\le r\le (40)^{2j}r\le |y|\lesssim r.
\end{equation} 
\end{proposition}

A similar conclusion for $\kappa\in (4,8)$ was proved in \cite[Section 2.3]{WuPolychromaticArmFKIsing}, the proof also works here with proper modifications. To be self-contained, we will give a complete proof. 
\begin{figure}[ht!]
\begin{center}
\includegraphics[width=0.5\textwidth]{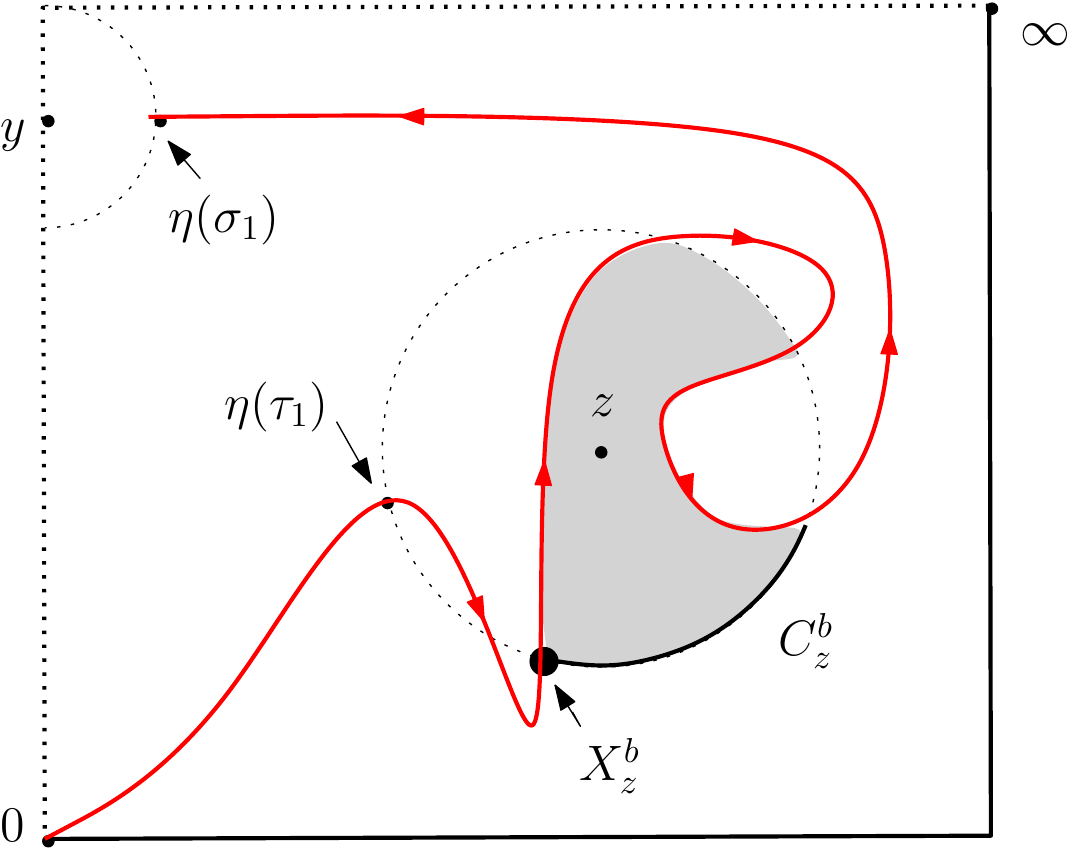}
\end{center}
\caption{\label{fig::sle_interior_def} The gray part is the connected component of $B(z,\eps)\setminus \eta[0,\sigma_1]$ that contains $z$, which is denoted by $C_z$. The bold part of $\partial C_z$ is $C_z^b$. The point $X_z^b$ is denoted in the figure. }
\end{figure}

\begin{lemma}\label{lem::sle_interior_derivative}
Fix $\kappa\in (0,4)$ and let $\eta$ be an $\SLE_{\kappa}$ in $\HH$ from 0 to $\infty$. Fix $z\in \HH$ with $|z|=1$. 
For $\eps>0$, let $\tau$ be the first time that $\eta$ hits $B(z, \eps)$. 
Define $\Theta_t=\arg(g_t(z)-W_t)$. For $\delta\in (0,1/16), R\ge 4$, define the event 
\[\LG=\{\tau<\infty, \Theta_{\tau}\in (\delta, \pi-\delta)\},\quad \LF=\LG\cap\{\eta[0,\tau]\subset B(0,R)\}.\]
For $\lambda\ge 0$, define 
\[\rho=\kappa/2-4-\sqrt{4\kappa\lambda+(\kappa/2-4)^2},\quad v(\lambda)=\frac{1}{2}-\frac{\kappa}{16}-\frac{\lambda}{2}+\frac{1}{8}\sqrt{4\kappa\lambda+(\kappa/2-4)^2}.\]
There exists a constant $R$ depending only on $\kappa$ and $z$ such that the following is true:
\[\eps^{v(\lambda)}\lesssim \E\left[|g'_{\tau}(z)|^{\lambda}\one_{\LF}\right]\le \E\left[|g'_{\tau}(z)|^{\lambda}\one_{\LG}\right]\lesssim \eps^{v(\lambda)}\delta^{-v(\lambda)-\rho^2/(2\kappa)},\]
where the constants in $\asymp$ depend on $\kappa, z$ and are uniform over $\eps, \delta$.
\end{lemma}
\begin{proof}
\cite[Lemma 2.11]{WuPolychromaticArmFKIsing}.
\end{proof}
Now we have decided the constant $R$ in Lemma \ref{lem::sle_interior_derivative}, and we will fix it in the following of the section. 
The conclusion for $\LE_2$ was proved in \cite[Proposition 4]{BeffaraDimension}, we will prove the conclusion for $\LE_{2j+2}$ for $j\ge 1$. We will need the following conclusion from Section \ref{sec::sle_boundary}. For $j\ge 1$, taking $\rho=0$ in Proposition \ref{prop::sle_boundary_alpha}, we have 
\[\alpha^+_{2j}=2j(2j+4-\kappa/2)/\kappa,\] 
\begin{equation}\label{eqn::sle_boundary_rhoequals0}
\PP\left[\LH^{\alpha}_{2j}(\eps, x, y, r)\right]\asymp x^{\alpha^+_{2j}-\alpha^+_{2j-1}}\eps^{\alpha^+_{2j-1}},\quad\text{provided } (40)^{2j}r\le |y|\lesssim r.
\end{equation}
Note that, since $\rho=0$, we may assume $v=x$ and we eliminate the force point in the definition of $\LH^{\alpha}$. 

\begin{proof}[Proof of (\ref{eqn::sle_interior}), Lower Bound] We will prove the lower bound for the probability of $\LE_{2j+2}$. 
Let $\eta$ be an $\SLE_{\kappa}$ in $\HH$ from 0 to $\infty$. Let $\tau$ be the first time that $\eta$ hits $B(z,\eps)$. 
Denote the centered conformal map $g_t-W_t$ by $f_t$ for $t\ge 0$. Recall that 
\[\LF=\{\eta[0,\tau]\subset B(0,R)\}.\]
Fix some $\delta>0$ and define 
\[\LG=\LF\cap \{\Theta_{\tau}\in (\delta, \pi-\delta)\}.\]

We run $\eta$ until the time $\tau$ and on the event $\LG$, by Koebe 1/4 theorem, we know that $f_{\tau}(B(z,\eps))$ contains the ball with center $w:=f_{\tau}(z)$ and radius $u:=\eps |f_{\tau}'(z)|/4$ and 
\[\arg(w)\in (\delta, \pi-\delta),\quad u\le \Im{w}\le 16u.\] 
We wish to apply (\ref{eqn::sle_boundary_rhoequals0}), however this ball is centered at $w=f_{\tau}(z)$ which does not satisfy the conditions in (\ref{eqn::sle_boundary_rhoequals0}). We will fix this problem by running $\eta$ for a little further and argue that there is positive chance that $\eta$ does the right thing. 

Let $\tilde{\eta}$ be the image of $\eta[\tau, \infty)$ under $f_{\tau}$. Let $\gamma$ be the broken line from 0 to $w$ and then to $-u+ui$ and let $A_u$ be the $u/4$-neighborhood of $\gamma$. Let $S_1$ be the first time that $\tilde{\eta}$ exits $A_u$ and let $S_2$ be the first time that $\tilde{\eta}$ hits the ball with center $-u+ui$ and radius $u/4$. By \cite[Lemma 2.5]{MillerWuSLEIntersection}, we know that $\PP[S_2<S_1]$ is bounded from below by positive constant depending only on $\kappa$ and $\delta$. On the event $\{S_2<S_1\}$, it is clear that there exist constants $x_{\delta}, c_{\delta}>0$ depending only on $\delta$ such that $f_{S_2}(B(z,\eps))$ contains the ball with center $x_{\delta}u$ and radius $c_{\delta}u$. 

Consider the image of $B(y,r)$ under $f_{S_2}$. On the event $\LF\cap \{S_2<S_1\}$, we know that the image of $B(y,r)$ under $f_{S_2}$ contains the ball with center $f_{S_2}(y)$ and radius $rf_{S_2}'(y)/4$ where 
\[2y\le f_{S_2}(y)\le y, \quad f_{S_2}'(y)\asymp 1.\]
Combining with (\ref{eqn::sle_boundary_rhoequals0}), we have
\[\PP\left[\LE_{2j+2}\cond \eta[0,S_2], \LG\cap \{S_2<S_1\}\right]\gtrsim (\eps |g_{\tau}'(z)|)^{\alpha_{2j}^+}.\]
Since $\{S_2<S_1\}$ has positive chance, we have 
\[\PP\left[\LE_{2j+2}\cond \eta[0,\tau], \LG\right]\gtrsim (\eps |g_{\tau}'(z)|)^{\alpha_{2j}^+}.\]
Therefore, by Lemma \ref{lem::sle_interior_derivative}, we have 
\[\PP\left[\LE_{2j+2}\right]\gtrsim \E\left[(\eps |g_{\tau}'(z)|)^{\alpha_{2j}^+}\one_{\LG}\right]\asymp\eps^{v(\alpha^+_{2j})+\alpha^+_{2j}}= \eps^{\alpha_{2j+2}},\]
where the constants in $\gtrsim$ and $\asymp$ depend only on $\kappa, z$ and $\delta$. This completes the proof. 
\end{proof}

\begin{lemma}\label{lem::sle_interior_upper_aux1}
Fix $\kappa\in (0,4)$ and let $\eta$ be an $\SLE_{\kappa}$ in $\HH$ from 0 to $\infty$. Fix $z\in\HH$ with $|z|=1$. Let $\Theta_t=\arg(g_t(z)-W_t)$. 
For $C\ge 16$, let $\xi$ be the first time that $\eta$ hits $\partial B(z, C\eps)$. For $\delta\in (0,1/16)$, define 
\[\LF=\{\xi<\infty, \Theta_{\xi}\in (\delta, \pi-\delta), \eta[0,\xi]\subset B(0,R)\}.\]
Then we have
\[\PP\left[\LE_{2j+2}(\eps, z, y, r)\cap \LF\right]\lesssim C^A\delta^{-B}\eps^{\alpha_{2j+2}},\quad \text{provided }y\le -20r, r\ge R.\]
where $A, B$ are some constants depending on $\kappa, j$, and the constant in $\lesssim$ depends only on $\kappa, j$, and is uniform over $\delta, C, \eps$.  
\end{lemma}
\begin{proof}
We run the curve up to time $\xi$ and let $f=g_{\xi}-W_{\xi}$. We have the following observations.
\begin{itemize}
\item By Lemma \ref{lem::image_insideball}, we know that $f(B(z,\eps))$ is contained in the ball with center $f(z)$ and radius $u:=4\eps|f'(z)|$. Applying Koebe 1/4 theorem to $f$, we have 
\begin{equation}\label{eqn::sle_interior_upperaux1_im}
C\eps |f'(z)|/4\le \Im{f(z)}\le 4C\eps|f'(z)|.
\end{equation}
Next, we argue that $f(B(z,\eps))$ is contained in the ball with center $|f(z)|\in \R$ and radius $8Cr/\delta$. Since $f((z,\eps))$ is contained in the ball with center $f(z)$ and radius $u$, it is clear that $f(B(z,\eps))$ is contained in the ball with center $|f(z)|$ with radius $u+2|f(z)|$. By (\ref{eqn::sle_interior_upperaux1_im}), we have 
\[Cu/16\le |f(z)|\sin\Theta_{\xi}\le Cu.\]
Since $\Theta_{\xi}\in (\delta, \pi-\delta)$, we know that, for $\delta>0$ small, we have $\sin\Theta_{\xi}\ge\delta/2$.
Thus, $Cu/16\le |f(z)|\le 2Cu/\delta$. Therefore, $f(B(z,\eps))$ is contained in the ball with center $|f(z)|$ with radius $8Cu/\delta$. In summary, we know that $f(B(z,\eps))$ is contained in the ball with center $|f(z)|$ and radius $32C\eps |f'(z)|/\delta$ where 
\[C\eps |f'(z)|/4\le |f(z)|\le 8C\eps|f'(z)|/\delta.\]
\item Consider $f(B(y,r))$. Since $\{\eta[0,\xi]\subset B(0,R)\}$ and $y\le -20r$ with $r\ge R$, we know that $f(B(y,r))$ is contained in the ball with center $f(y)$ and radius $4rf'(y)$ where 
\[2y\le f(y)\le y, \quad f'(y)\asymp 1.\]
\end{itemize}
Combining these two facts with (\ref{eqn::sle_boundary_rhoequals0}), we have 
\[\PP\left[\LE_{2j+2}(\eps, z, y, r)\cond \eta[0,\xi], \LF\right]\lesssim \left(C\eps|f'(z)|/\delta\right)^{\alpha^+_{2j}},\]
where the constant in $\lesssim$ depends only on $\kappa$ and is independent of $C, \eps, \delta$. 
Thus, by Lemma \ref{lem::sle_interior_derivative}, we have 
\[\PP\left[\LE_{2j+2}(\eps, z, y, r)\cap \LF\right]\lesssim C^A\delta^{-B}\eps^{\alpha_{2j+2}},\]
where $A, B$ are some constants depending on $\kappa, j$. 
This completes the proof. 
\end{proof}

\begin{lemma}\label{lem::sle_interior_upper_aux2}
Fix $\kappa\in (0,8)$ and let $\eta$ be an $\SLE_{\kappa}$ in $\HH$ from 0 to $\infty$. Fix $z\in\HH$ with $|z|=1$. Let $T_z$ be the first time that $\eta$ swallows $z$ and set $\Theta_t=\arg(g_t(z)-W_t)$.
Take $n\in\N$ such that $B(z, 16\eps 2^n)$ is contained in $\HH$. For $1\le m\le n$, let $\xi_m$ be the first time that $\eta$ hits $B(z, 16\eps 2^{n-m+1})$. Note that $\xi_1, ..., \xi_n$ is an increasing sequence of stopping times and $\xi_1$ is the first time that $\eta$ hits $B(z, 16\eps 2^n)$ and $\xi_n$ is the first time that $\eta$ hits $B(z, 32\eps)$. For $1\le m\le n$, for $\delta>0$, define 
\[\LF_m=\{\xi_m<T_z, \Theta_{\xi_m}\not\in (\delta, \pi-\delta)\}\]
There exists a function $p:(0,1)\to [0,1]$ with $p(\delta)\downarrow 0$ as $\delta\downarrow 0$ such that 
\[\PP\left[\cap_{1}^n \LF_m\right]\le p(\delta)^n.\]
\end{lemma}
\begin{proof}
\cite[Lemma 2.13]{WuPolychromaticArmFKIsing}.
\end{proof}
\begin{proof}[Proof of (\ref{eqn::sle_interior}), Upper Bound]
Assume the same notations as in Lemma \ref{lem::sle_interior_upper_aux2}. Recall that 
\[\LF=\{\eta[0,\tau_1]\subset B(0,R)\}.\]
By Lemma \ref{lem::sle_interior_upper_aux1}, we have, for $1\le m\le n$
\[\PP\left[\LE_{2j+2}\cap\LF\cap \LF_m^c\right]\lesssim 2^{nA}\delta^{-B}\eps^{\alpha_{2j+2}},\]
where $A, B$ are some constants depending on $\kappa$. Combining with Lemma \ref{lem::sle_interior_upper_aux2}, we have, for any $n$ and $\delta>0$, 
\[
\PP\left[\LE_{2j+2}(\eps, z, y, r)\cap \LF\right]
\lesssim n2^{nA}\delta^{-B}\eps^{\alpha_{2j+2}} +p(\delta)^n,\]
where $p(\delta)\downarrow 0$ as $\delta\downarrow 0$. This implies the conclusion. 
\end{proof}

%% file: tex/ising.tex
We focus on the square lattice $\Z^2$. Two vertices $x=(x_1, x_2)$ and $y=(y_1, y_2)$ are neighbors if $|x_1-y_1|+|x_2-y_2|=1$, and we write $x\sim y$.  
We denote by $\Lambda_n(x)$ the box centered at $x$:
\[\Lambda_n(x)=x+[-n,n]^2,\quad \Lambda_n=\Lambda_n(0).\]
Let $\Omega$ be a finite subset of $\Z^2$, and the edge-set of $\Omega$ consists of all edges of $\Z^2$ that links two vertices of $\Omega$. The boundary of $\Omega$ is defined to be $\partial \Omega=\{e=(x,y): x\sim y, x\in \Omega, y\not\in \Omega\}$. We sometimes identify a boundary edge $(x,y)$ with one of its endpoint. Two vertices $x=(x_1, x_2)$ and $y=(y_1, y_2)$ are $\star$-neighbors if $\max\{|x_1-y_1|, |x_2-y_2|\}=1$. With this definition, each vertex has eight $\star$-neighbors instead of four. 

The Ising model with free boundary conditions is a random assignment $\sigma\in \{\ominus, \oplus\}^{\Omega}$ of spins $\sigma_x\in \{\ominus, \oplus\}$, where $\sigma_x$ denotes the spin at the vertex $x$. The Hamiltonian of the Ising model is defined by 
\[H^{\free}_{\Omega}(\sigma)=-\sum_{x\sim y}\sigma_x\sigma_y.\] 
The Ising measure is the Boltzmann measure with Hamiltonian $H^{\free}_{\Omega}$ and inverse-temperature $\beta>0$: 
\[\mu^{\free}_{\beta,\Omega}[\sigma]=\frac{\exp(-\beta H^{\free}_{\Omega}(\sigma))}{Z^{\free}_{\beta, \Omega}},\quad\text{where }Z^{\free}_{\beta, \Omega}=\sum_{\sigma}\exp(-\beta H^{\free}_{\Omega}(\sigma)).\]

For a graph $\Omega$ and $\tau\in \{\ominus, \oplus\}^{\Z^2}$, one may also define the Ising model with boundary conditions $\tau$ by the Hamiltonian
\[H^{\tau}_{\Omega}(\sigma)=-\sum_{x\sim y, \{x,y\}\cap \Omega\neq\emptyset}\sigma_x\sigma_y,\quad \text{if }\sigma_x=\tau_x, \forall x\not\in \Omega.\]

\textit{Dobrushin domains} are discrete analogue of simply connected domains with two marked points on their boundary. Suppose that $(\Omega, a, b)$ is a Dobrushin domain. Assume that $\partial \Omega$ can be divided into two $\star$-connected paths from $a$ to $b$ (counterclockwise) and from $b$ to $a$. Several boundary conditions will be of particular interest in this paper. 
\begin{itemize}
\item We denote by $\mu^{\free}$ for free boundary conditions. We denote by $\mu^{\oplus}$ (resp. $\mu^{\ominus}$) for the boundary conditions that $\tau_x=\oplus$ for all $x$ (resp. $\tau_x=\ominus$ for all $x$). 
\item $(\ominus\oplus)$ boundary conditions: $\oplus$ along $\partial \Omega$ from $a$ to $b$, and $\ominus$ along $\partial \Omega$ from $b$ to $a$. This boundary condition is also called Dobrushin boundary condition, or domain-wall boundary condition. 
\item $(\ominus\free)$ boundary conditions: $\free$ along $\partial \Omega$ from $a$ to $b$, and $\ominus$ along $\partial \Omega$ from $b$ to $a$.
\end{itemize}

\begin{proposition}[Domain Markov Property]\label{prop::ising_domainmarkov}
Let $\Omega\subset \Omega'$ be two finite subsets of $\Z^2$. Let $\tau\in \{\ominus, \oplus\}^{\Z^2}$ and $\beta>0$. Let $X$ be a random variable which is measurable with respect to vertices in $\Omega$. Then we have 
\[\mu^{\tau}_{\beta, \Omega'}[X\cond \sigma_x=\tau_x, \forall x\in \Omega'\setminus \Omega]=\mu^{\tau}_{\beta, \Omega}[X].\] 
\end{proposition}

The set $\{\ominus, \oplus\}^{\Omega}$ is equipped with a partial order: $\sigma\le \sigma'$ if $\sigma_x\le \sigma_x'$ for all $x\in \Omega$. A random variable $X$ is increasing if $\sigma\le \sigma'$ implies $X(\sigma)\le X(\sigma')$. An event $\LA$ is increasing if $\one_{\LA}$ is increasing. 
\begin{proposition}[FKG inequality]\label{prop::ising_fkg}
Let $\Omega$ be a finite subset and $\tau$ be boundary conditions, and $\beta>0$. For any two increasing events $\LA$ and $\LB$, we have 
\[\mu^{\tau}_{\beta, \Omega}[\LA\cap\LB]\ge \mu^{\tau}_{\beta, \Omega}[\LA]\mu^{\tau}_{\beta, \Omega}[\LB].\]
\end{proposition}
\begin{proof}
\cite[Chapter 3, Theorem 3.32]{VelenikFriedliStatisticalMechnics}.
\end{proof}
As a consequence of FKG inequality, we have the comparison between boundary conditions. For boundary conditions $\tau_1\le \tau_2$ and an increasing event $\LA$, we have
\begin{equation}\label{eqn::ising_boundary_comparison}
\mu^{\tau_1}_{\beta, \Omega}[\LA]\le \mu^{\tau_2}_{\beta, \Omega}[\LA]. 
\end{equation}

Ising model with inverse-temperature $\beta>0$ is related to random-cluster model with parameters $(p,2)$ through Edwards-Sokal coupling, thus the critical value $p_c(2)$ for the random-cluster model gives the critical value of $\beta:$
\[\beta_c=\frac{1}{2}\log(1+\sqrt{2}).\] 

A discrete topological rectangle $(\Omega, a, b, c, d)$ is a bounded simply-connected subdomains of $\Z^2$ with four marked boundary points. The four points are in counterclockwise order and $(ab)$ denotes the arc of $\partial \Omega$ from $a$ to $b$. We denote by $d_{\Omega}((ab), (cd))$ the discrete extermal distance between $(ab)$ and $(cd)$ in $\Omega$, see \cite[Section 6]{ChelkakRobustComplexAnalysis}. The discrete extremal distance is uniformly comparable to and converges to its continuous counterpart--- the classical extremal distance. The rectangle $(\Omega, a, b, c, d)$ is crossed by $\oplus$ in an Ising configuration $\sigma$ if there exists a path of $\oplus$ going from $(ab)$ to $(cd)$ in $\Omega$. We denote this event by $(ab)\overset{\oplus}{\longleftrightarrow}(cd)$. We have the following RSW-type estimate on the crossing probability at critical. 
\begin{proposition}
[RSW for topological rectangle]\label{prop::ising_rsw}
For each $L>0$ there exists $c(L)>0$ such that the following holds: for any topological rectangle $(\Omega, a, b, c, d)$ with $d_{\Omega}((ab), (cd))\le L$, 
\[\mu^{\text{mixed}}_{\beta_c, \Omega}\left[(ab)\overset{\oplus}{\longleftrightarrow}(cd)\right]\ge c(L),\]
where the boundary conditions are $\free$ on $(ab)\cup(cd)$ and $\ominus$ on $(bc)\cup(da)$. 
\end{proposition}
\begin{proof}
\cite[Corollary 1.7]{ChelkakDuminilHonglerCrossingprobaFKIsing}.
\end{proof}

As a consequence of Propositions \ref{prop::ising_domainmarkov} to \ref{prop::ising_rsw}, we have the following space mixing property at critical.
\begin{corollary}\label{cor::ising_spacemixing}
There exists $\alpha>0$ such that for any $2k\le n$, for any event $\LA$ depending only on edges in $\Lambda_k$, and for any boundary conditions $\tau, \xi$, we have 
\[|\mu^{\tau}_{\beta_c, \Lambda_n}[\LA]-\mu^{\xi}_{\beta_c, \Lambda_n}[\LA]|\le \left(\frac{k}{n}\right)^{\alpha}\mu^{\tau}_{\beta_c, \Lambda_n}[\LA].\]
In particular, this implies that, 
for any boundary conditions $\tau$, for any $2k\le n\le m$, for any event $\LA$ depending only on vertices of $\Lambda_k$, and for any event $\LB$ depending only on vertices of $\Lambda_m\setminus\Lambda_n$, we have 
\[|\mu^{\tau}_{\beta_c, \Lambda_m}[\LA\cap\LB]-\mu^{\tau}_{\beta_c, \Lambda_m}[\LA]\mu^{\tau}_{\beta_c, \Lambda_m}[\LB]|\le \left(\frac{k}{n}\right)^{\alpha}\mu^{\tau}_{\beta_c, \Lambda_m}[\LA]\mu^{\tau}_{\beta_c, \Lambda_m}[\LB].\] 
\end{corollary}

%% file: tex/qm.tex
Fix $n<N$ and the annulus $\Lambda_N\setminus \Lambda_n$, a simple path of $\oplus$ or of $\ominus$ connecting $\partial \Lambda_n$ to $\partial \Lambda_N$ is called an \textit{arm}. Fix an integer $j\ge 1$ and $\omega=(\omega_1, ..., \omega_j)\in \{\ominus, \oplus\}^j$. For $n<N$, define $\LA_{\omega}(n, N)$ to be the event that there are $j$ disjoint arms $(\gamma_k)_{1\le k\le j}$ connecting $\partial \Lambda_n$ to $\partial \Lambda_N$ in the annulus $\Lambda_N\setminus \Lambda_n$ which are of types $(\omega_k)_{1\le k\le j}$, where we identify two sequences $\omega$ and $\omega'$ if they are the same up to cyclic permutation and the arms are indexed in clockwise order. For each $j\ge 1$, there exists a smallest integer $n_0(j)$ such that, for all $N\ge n_0(j)$, we have $\LA_{\omega}(n_0(j), N)\neq\emptyset$.

\begin{proposition}\label{prop::qm_interior_alternating}
Assume that $\omega$ is alternating with even length. 
For all $n_0(j)\le n_1<n_2<n_3\le m/2$, and for all boundary conditions $\tau$, we have 
\[\mu^{\tau}_{\beta_c, \Lambda_m}\left[\LA_{\omega}(n_1, n_3)\right]\asymp \mu^{\tau}_{\beta_c, \Lambda_m}\left[\LA_{\omega}(n_1, n_2)\right]\mu^{\tau}_{\beta_c, \Lambda_m}\left[\LA_{\omega}(n_2, n_3)\right],\]
where the constants in $\asymp$ are uniform over $n_1, n_2, n_3, m$ and $\tau$.
\end{proposition}   

Proposition \ref{prop::qm_interior_alternating} is called the quasi--multiplicativity. 
We will introduce several auxiliary subevents of $\LA_{\omega}(n,N)$ which are both important for the proof of Proposition \ref{prop::qm_interior_alternating} and also important for us to derive the arm exponents of Ising. Fix $\omega=(\omega_1, ...,\omega_j)\in \{\ominus, \oplus\}^j$. Fix some $\delta>0$ small. Suppose $Q=[-1,1]^2$ is the unit square. A landing sequence $(I_k)_{1\le k\le j}$ is a sequence of disjoint sub-intervals on $\partial Q$ in clockwise order. We denote by $z(I_k)$ the center of $I_k$. We say $(I_k)_{1\le k\le j}$ is \textit{$\delta$-separated} if 
\begin{itemize}
\item the intervals are at distance at least $2\delta$ from each other, and they are at distance at least $2\delta$ from the four corners of $\partial Q$ 
\item  for each $I_k$, the length of $I_k$ is at least $2\delta$.
\end{itemize}
We say that two sets are $\omega_k$-connected if there is a path of type $\omega_k$ connecting them.  
Fix two $\delta$-separated landing sequences $(I_k)_{1\le k\le j}$ and $(I'_k)_{1\le k\le j}$. We say that the arms $(\gamma_k)_{1\le k\le j}$ are \textit{$\delta$-well-separated} with landing sequence $(I_k)_{1\le k\le j}$ on $\partial \Lambda_n$ and landing sequence $(I_k')_{1\le k\le j}$ on $\partial \Lambda_N$ if 
\begin{itemize}
\item for each $k$, the arm $\gamma_k$ connects $nI_k$ to $NI_k'$;
\item for each $k$, the arm $\gamma_k$ can be $\omega_k$-connected to distance $\delta n$ of $\partial \Lambda_n$ inside $\Lambda_{\delta n}(z(I_k))$;
\item for each $k$, the arm $\gamma_k$ can be $\omega_k$-connected to distance $\delta N$ of $\partial \Lambda_N$ inside $\Lambda_{\delta N}(z(I'_k))$.
\end{itemize}
We denote this event by 
\[\LA_{\omega}^{I/I'}(n,N).\]
\begin{lemma}\label{lem::wellseparated_comparable}
Fix $j\ge 1$ and $\delta>0$ and two $\delta$-separated landing sequences $(I_k)_{1\le k\le j}$ and $(I'_k)_{1\le k\le j}$. Assume that $\omega$ is alternating with length $2j$. 
For all $n<N\le m/2$ such that $\LA_{\omega}^{I/I'}(n, N)$ is not empty, and for all boundary conditions $\tau$, we have 
\[\mu^{\tau}_{\beta_c, \Lambda_m}\left[\LA_{\omega}^{I/I'}(n, N)\right]\asymp \mu^{\tau}_{\beta_c,\Lambda_m}\left[\LA_{\omega}(n, N)\right],\]
where the constants in $\asymp$ depend only on $\delta$.
\end{lemma}
We have similar results for the boundary arm events. Denote by 
\[\Lambda_n^+(x)=[-n,n]\times [0,n]+x,\quad \Lambda_n^+=\Lambda^+_n(0).\] 
We consider the arm events in the semi-annulus $\Lambda^+_N\setminus \Lambda_n^+$ and extend the definition of arm events and arm events with landing sequences in the obvious way, and denote them as 
\[\LA^{+}_{\omega}(n,N),\quad \LA^{+, I/I}_{\omega}(n,N). \] 

We need to restrict to the cases that the arms together with the boundary conditions are alternating. Precisely, in the statements of Proposition \ref{prop::qm_boundary_alternating} and Lemma \ref{lem::ws_boundary_comparable}, we restrict to the cases where the arm patterns and the boundary conditions are listed in Theorem \ref{thm::ising_boundary}. 
\begin{proposition}\label{prop::qm_boundary_alternating} For all $n_0^+(j)\le n_1<n_2<n_3\le m/2$, we have 
\[\mu^{\tau}_{\beta_c, \Lambda^+_m}\left[\LA^+_{\omega}(n_1, n_3)\right]\asymp \mu^{\tau}_{\beta_c, \Lambda^+_m}\left[\LA^+_{\omega}(n_1, n_2)\right]\mu^{\tau}_{\beta_c, \Lambda^+_m}\left[\LA^+_{\omega}(n_2, n_3)\right],\]
where the constants in $\asymp$ are uniform over $n_1, n_2, n_3$ and $m$. 
\end{proposition}

\begin{lemma}\label{lem::ws_boundary_comparable}
Fix $j\ge 1$, $\delta>0$ and two $\delta$-separated landing sequences $(I_k)_{1\le k\le j}$ and $(I'_k)_{1\le k\le j}$. 
For all $n<N\le m/2$ such that $\LA_{\omega}^{+, I/I'}(n, N)$ is not empty, we have 
\[\mu^{\tau}_{\beta_c, \Lambda^+_m}\left[\LA_{\omega}^{+, I/I'}(n, N)\right]\asymp \mu^{\tau}_{\beta_c,\Lambda^+_m}\left[\LA^+_{\omega}(n, N)\right],\]
where the constants in $\asymp$ depend only on $\delta$.
\end{lemma}

We do not plan to give the proofs of the quasi-multiplicativity in this paper, because the proof is exactly the same as the proof of the quasi-multiplicativity for FK-Ising model proved in \cite{ChelkakDuminilHonglerCrossingprobaFKIsing} where all the ingredients needed in the proof are the ones listed in Section \ref{subsec::ising_properties}. 

%% file: tex/ising_interface_cvg.tex
The \textit{dual square lattice} $(\Z^2)^*$ is the dual graph of $\Z^2$. The vertex set is $(1/2, 1/2)+\Z^2$ and the edges are given by nearest neighbors.  The vertices and edges of $(\Z^2)^*$ are called dual-vertices and dual-edges. In particular, for each edge $e$ of $\Z^2$, it is associated to a dual edge, denoted by $e^*$, that it crosses $e$ in the middle. 
For a finite subgraph $G$, we define $G^*$ to be the subgraph of $(\Z^2)^*$ with edge-set $E(G^*)=\{e^*: e\in E(G)\}$ and vertex set given by the end-points of these dual-edges. The \textit{medial lattice} $(\Z^2)^{\diamond}$ is the graph with the centers of edges of $\Z^2$ as vertex set, and edges connecting nearest vertices. This lattice is a rotated and rescaled version of $\Z^2$, see Figure \ref{fig::medial_lattice}. The vertices and edges of $(\Z^2)^{\diamond}$ are called medial-vertices and medial-edges. We identify the faces of $(\Z^2)^{\diamond}$ with the vertices of $\Z^2$ and $(\Z^2)^*$. A face of $(\Z^2)^{\diamond}$ is said to be black if it corresponds to a vertex of $\Z^2$ and white if it corresponds to a vertex of $(\Z^2)^*$. 

\begin{figure}[ht!]
\begin{subfigure}[b]{0.33\textwidth}
\begin{center}
\includegraphics[width=0.8\textwidth]{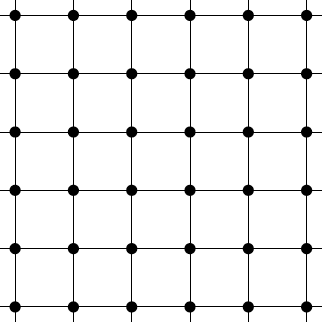}
\end{center}
\caption{The square lattice. }
\end{subfigure}
\begin{subfigure}[b]{0.33\textwidth}
\begin{center}\includegraphics[width=0.8\textwidth]{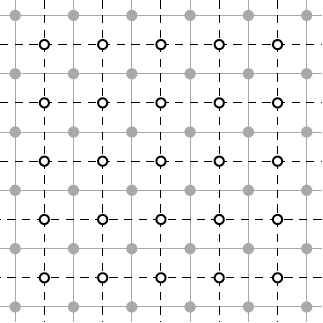}
\end{center}
\caption{The dual square lattice.}
\end{subfigure}
\begin{subfigure}[b]{0.33\textwidth}
\begin{center}
\includegraphics[width=0.8\textwidth]{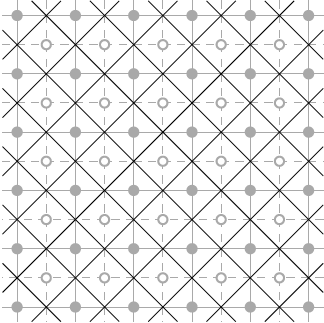}
\end{center}
\caption{The medial lattice. }
\end{subfigure}
\caption{\label{fig::medial_lattice} The lattices. }
\end{figure}

\begin{figure}[ht!]
\begin{center}
\includegraphics[width=0.8\textwidth]{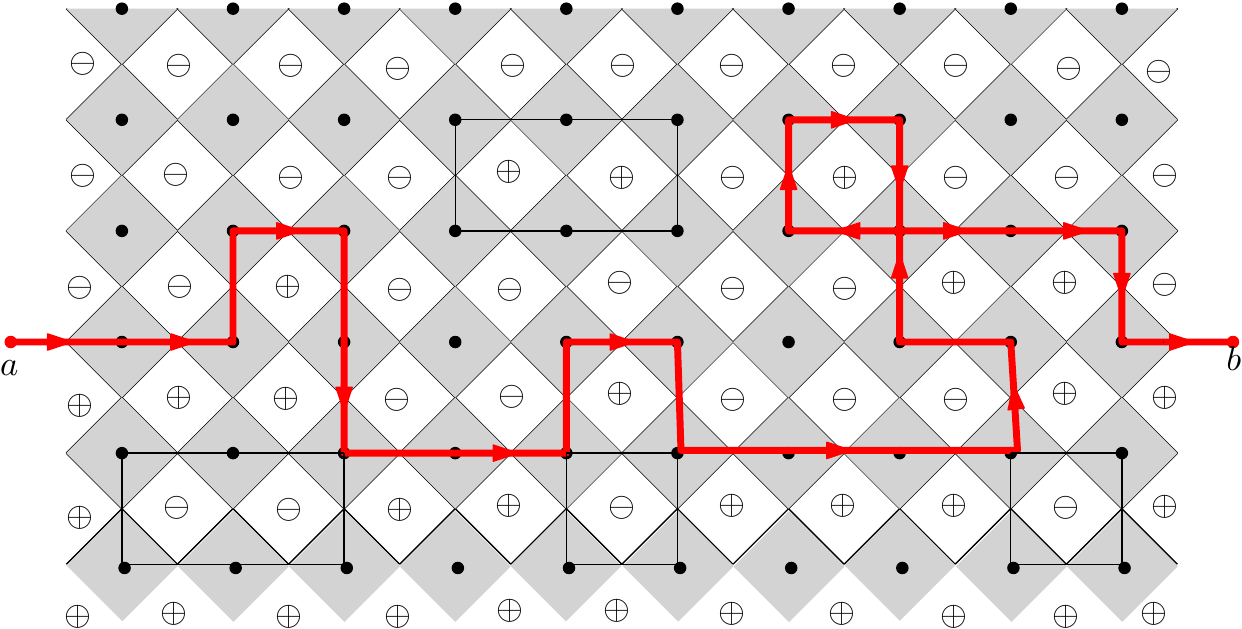}
\end{center}
\caption{\label{fig::ising_interface} The Ising interface. }
\end{figure}

For $u>0$, we consider the rescaled square lattice $u\Z^2$. The definitions of dual lattice, medial lattice and  Dobrushin domains extend to this context, and they will be denoted by $(\Omega_u, a_u, b_u)$, $(\Omega^*_u, a^*_u, b^*_u)$, $(\Omega^{\diamond}_u, a^{\diamond}_u, b^{\diamond}_u)$ respectively. 
Consider the critical Ising model on $(\Omega^*_u, a^*_u, b^*_u)$. The boundary $\partial \Omega^*_u$ is divided into two parts $(a^*_ub^*_u)$ and $(b^*_ua^*_u)$. We fix the boundary conditions to be $\ominus$ on $(b^*_ua^*_u)$ and $\oplus$ on $(a^*_ub^*_u)$, or $\ominus$ on $(b^*_ua^*_u)$ and $\free$ on $(a^*_ub^*_u)$. Define the \textit{interface} as follows. It starts from $a_u^{\diamond}$, lies on the primal lattice and turns at every vertex of $\Omega_u$ is such a way that it has always dual vertices with spin $\ominus$ on its left and $\oplus$ on its right. If there is an indetermination when arriving at a vertex (this may happen on the square lattice), turn left. See Figure \ref{fig::ising_interface}.

Let $(\Omega, a, b)$ be a simply connected domain with two marked points on its boundary. Consider a sequence of Dobrushin domains $(\Omega_u, a_u, b_u)$. We say that $(\Omega_u, a_u, b_u)$ converges to $(\Omega, a, b)$ in the \textit{Carath\'eodory sense} if 
\[f_u\to f\quad\text{on any compact subset }K\subset\HH,\]
where $f_u$ (resp. $f$) is the unique conformal map from $\HH$ to $\Omega_u$ (resp. $\Omega$) satisfying $f_u(0)=a_u, f_u(\infty)=b_u$ and $f_u'(\infty)=1$ (resp. $f(0)=a, f(\infty)=b, f'(\infty)=1$). 

Let $X$ be the set of continuous parameterized curves and $d$ be the distance on $X$ defined for $\eta_1: I\to \C$ and $\eta_2: J\to \C$ by 
\[d(\eta_1, \eta_2)=\min_{\varphi_1: [0,1]\to I, \varphi_2:[0,1]\to J}\sup_{t\in [0,1]}|\eta_1(\varphi_1(t))-\eta_2(\varphi_2(t))|,\]
where the minimization is over increasing bijective functions $\varphi_1, \varphi_2$. Note that $I$ and $J$ can be equal to $\R_+\cup\{\infty\}$. The topology on $(X, d)$ gives rise to a notion of weak convergence for random curves on $X$. 

\begin{theorem}\label{thm::ising_cvg_minusplus}
Let $\Omega$ be a simply connected domain with two marked points $a$ and $b$ on its boundary. Let $(\Omega^{\diamond}_u, a^{\diamond}_u, b^{\diamond}_u)$ be a family of Dobrushin domains converging to $(\Omega, a, b)$ in the Carath\'eodory sense. The interface of the critical Ising model in $(\Omega^*_u, a^*_u, b^*_u)$ with $(\ominus\oplus)$ boundary conditions converges weakly to $\SLE_{3}$ as $u\to 0$.
\end{theorem}
\begin{proof}
\cite{CDCHKSConvergenceIsingSLE}.
\end{proof}

\begin{theorem}\label{thm::ising_cvg_minusfree}
Let $\Omega$ be a simply connected domain with two marked points $a$ and $b$ on its boundary.  Let $(\Omega^{\diamond}_u, a^{\diamond}_u, b^{\diamond}_u)$ be a family of Dobrushin domains converging to $(\Omega, a, b)$ in the Carath\'eodory sense. The interface of the critical Ising model in $(\Omega^*_u, a^*_u, b^*_u)$ with $(\ominus\free)$ boundary conditions converges weakly to $\SLE_{3}(-3/2)$ as $u\to 0$.
\end{theorem}
\begin{proof}
It is proved in \cite{HonglerKytolaIsingFree, BenoistDuminilHonglerIsingFree} that the interface with $(\free\free)$ boundary conditions converges weakly to $\SLE_3(-3/2;-3/2)$ as $u\to 0$. The same proof works here. 
\end{proof}

%% file: tex/ising_arms.tex
\begin{figure}[ht!]
\begin{subfigure}[b]{0.5\textwidth}
\begin{center}
\includegraphics[width=0.8\textwidth]{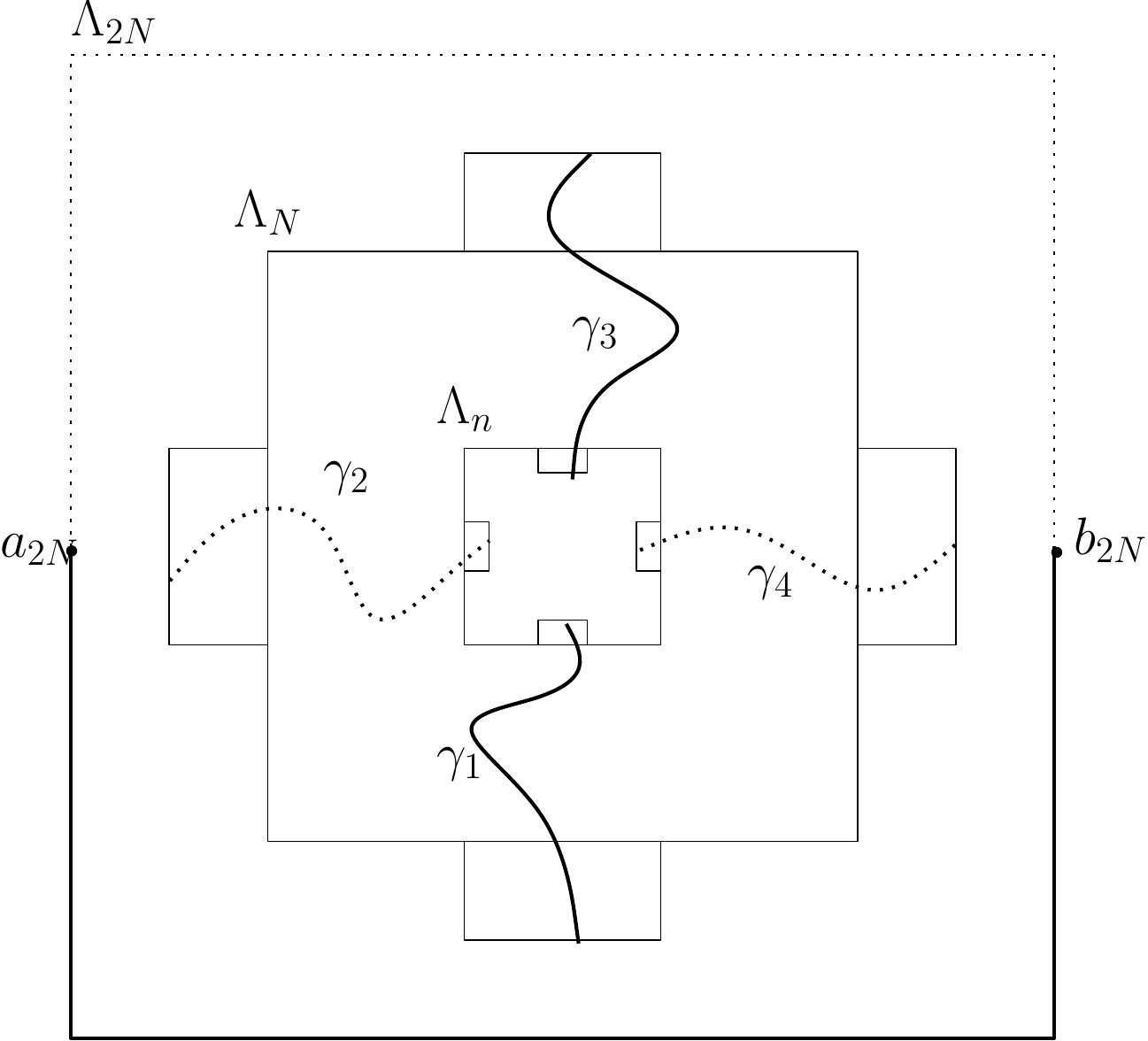}
\end{center}
\caption{$\LA^{I/I}(n,N)$ is the well-separated arm event. }
\end{subfigure}
\begin{subfigure}[b]{0.5\textwidth}
\begin{center}\includegraphics[width=0.8\textwidth]{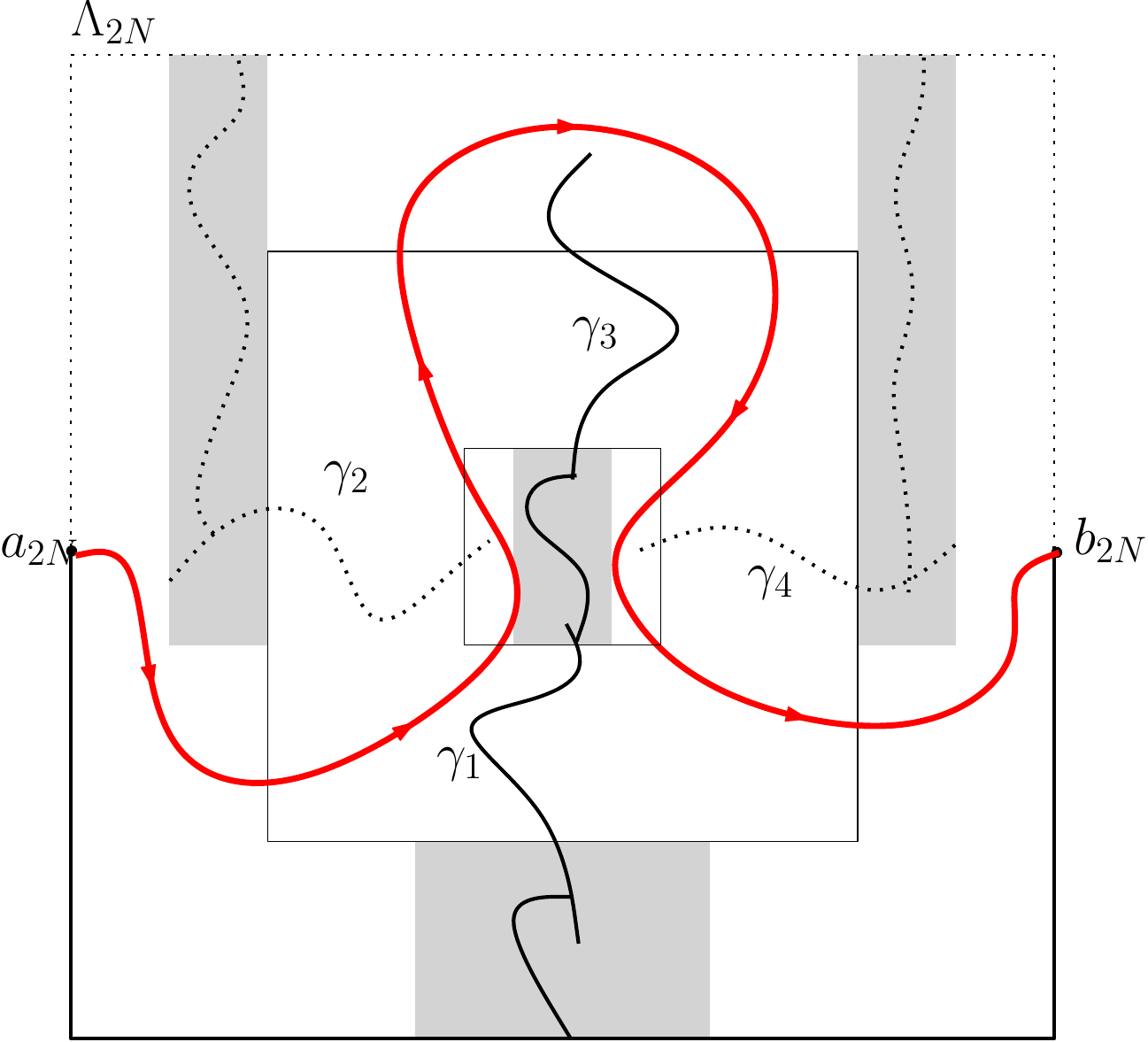}
\end{center}
\caption{The four gray parts are $R_1$ to $R_4$ respectively. }
\end{subfigure}
\caption{\label{fig::ising_interior} The explanation of the proof of Theorem \ref{thm::ising_interior}.}
\end{figure}

\begin{proof}[Proof of Theorem \ref{thm::ising_interior}]
We only give the proof for $\alpha_4$ and the other cases can be proved similarly. Consider $\Lambda_m$ with two boundary points $a_m=(-m,0)$ and $b_m=(m,0)$. Fix the $(\ominus\oplus)$ boundary condition: the vertices along $\partial\Lambda_m$ from $b_m$ to $a_m$ (counterclockwise) are $\oplus$ and the vertices from $a_m$ to $b_m$ are $\ominus$. Since we fix $\beta=\beta_c$ and the boundary conditions, and $\omega=(\oplus\ominus\oplus\ominus)$, we eliminate them from the notations. We will prove that, for $n<N\le m/2$, 
\begin{equation}\label{eqn::ising_4arms}
\mu_{\Lambda_m}[\LA(n,N)]=N^{-\alpha_4+o(1)},\quad \text{as }N\to \infty.
\end{equation}

Fix the landing sequence $I=(I_1, I_2, I_3, I_4)$ where 
\[I_1=[-1/2, 1/2]\times\{-1\}, \quad I_2=\{-1\}\times [-1/2, 1/2],\quad I_3=[-1/2, 1/2]\times\{1\},\quad I_4=\{1\}\times [-1/2, 1/2].\] 
Recall that $\LA^{I/I}(n,N)$ is the $1/8$-well-separated arm events with the landing sequence $nI$ on $\partial \Lambda_n$ and $NI$ on $\partial\Lambda_N$. The four arms in $\LA(n,N)$ are denoted by $(\gamma_1, \gamma_2, \gamma_3, \gamma_4)$ where $\gamma_1$ and $\gamma_3$ are $\oplus$ and $\gamma_2$ and $\gamma_4$ are $\ominus$. Consider critical Ising model in $\Lambda_{2N}$. Let $R_1$ to be the rectangle $[-3N/4, 3N/4]\times [-2N, -N]$, and define $\LC^{\oplus}_1$ to be the event that $\gamma_1$ is connected by path of $\oplus$ in $R_1$ to the bottom of $R_1$. Let $R_2$ to be the rectangle $[-9N/8, -N]\times[-N/2,2N]$, and define $\LC_2^{\ominus}$  to be the event that $\gamma_2$ is connected by path of $\ominus$ in $R_2$ to the top of $R_2$. Let $R_3$ be the rectangle $[-3n/4, 3n/4]\times[-n,n]$, and define $\LC_3^{\oplus}$ to be the event that $\gamma_3$ is connected to $\gamma_1$ by path of $\oplus$ in $R_3$. Let $R_4$ be the rectangle $[N, 9N/8]\times[-N/2,2N]$, and define $\LC_4^{\ominus}$ to be the event that $\gamma_4$ is connected by path of $\ominus$ in $R_4$ to the top of $R_4$. See Figure \ref{fig::ising_interior}.

By (\ref{eqn::ising_boundary_comparison}), Proposition \ref{prop::ising_rsw} and Corollary \ref{cor::ising_spacemixing}, we could prove
\begin{equation}\label{eqn::ising_interior_aux1}
\mu_{\Lambda_{2N}}\left[\LA^{I/I}(n,N)\right]\asymp \mu_{\Lambda_{2N}}\left[\LA^{I/I}(n,N)\cap\LC^{\oplus}_1\cap\LC_2^{\ominus}\cap\LC^{\oplus}_3\cap\LC^{\ominus}_4\right],
\end{equation}
where the constants in $\asymp$ are uniform over $n,N$. 

Let $\PP_N$ be the probability measure $\mu_{\Lambda_{2N}}$ where the square lattice is scaled by $1/N$ and let $\PP_{\infty}$ be the law of $\SLE_{3}$ in $[-2,2]\times[-2,2]$ from $(-2,0)$ to $(2,0)$. On the event $\LA^{I/I}(n,N)\cap\LC^{\oplus}_1\cap\LC_2^{\ominus}\cap\LC^{\oplus}_3\cap\LC^{\ominus}_4$, consider the interface $\eta$ from $a_{2N}$ to $b_{2N}$. Let $\tau_1$ be the first time that $\eta$ hits $\partial\Lambda_n$. The event $\LC_1^{\oplus}\cap\LC_2^{\ominus}$ guarantees that $\eta[0,\tau_1]$ is bounded away from the target $b_{2N}$. The event $\LC^{\oplus}_3$ guarantees that, after $\tau_1$, the path $\eta$ hits the neighborhood of $(0,2N)$ at some time $\sigma_1$. The event $\LC^{\ominus}_4$ guarantees that, after $\sigma_1$, the path $\eta$ hits $\partial\Lambda_n$ again. Therefore, by (\ref{eqn::sle_interior}), we have, for $\eps>0$,
\[\limsup_{N\to \infty}\PP_N\left[\LA^{I/I}(\eps N,N)\cap\LC^{\oplus}_1\cap\LC_2^{\ominus}\cap\LC^{\oplus}_3\cap\LC^{\ominus}_4\right]\le \eps^{\alpha_4+o(1)}\le\liminf_{N\to\infty}\PP_{N}[\LA(\eps N, N)].\]
Combining with Lemma \ref{lem::wellseparated_comparable} and (\ref{eqn::ising_interior_aux1}), we have 
\[\liminf_{N\to\infty}\PP_{N}[\LA(\eps N, N)]\asymp \limsup_{N\to\infty}\PP_{N}[\LA(\eps N, N)]\asymp \eps^{\alpha_4+o(1)}.\]
By Corollary \ref{cor::ising_spacemixing}, we know that 
\begin{equation}\label{eqn::ising_interior_aux2}
\liminf_{N\to \infty}\mu_{\Lambda_m}\left[\LA(\eps N, N)\right]\asymp \limsup_{N\to\infty}\mu_{\Lambda_m}\left[\LA(\eps N, N)\right]\asymp \eps^{\alpha_4+o(1)},
\end{equation} 
where the constants in $\asymp$ are uniform over $\eps$ and $m\ge 2N$. 

Suppose $N=n\eps^{-K}$ for some integer $K$. By Proposition \ref{prop::qm_interior_alternating}, for $m\ge 2N$, we have 
\[\mu_{\Lambda_m}\left[\LA(n, N)\right]\le C^K\Pi_{j=1}^K\mu_{\Lambda_m}\left[\LA(n\eps^{-j+1}, n\eps^{-j})\right],\]
where $C$ is some universal constant. Thus
\[\frac{\log\mu_{\Lambda_m}\left[\LA(n, N)\right]}{\log N}\le \frac{K\log C}{\log N}+\frac{1}{\log N}\sum_{j=1}^K\log \mu_{\Lambda_m}\left[\LA(n\eps^{-j-1}, n\eps^{-j})\right].\]
By (\ref{eqn::ising_interior_aux2}), we have 
\[\limsup_{j\to \infty}\mu_{\Lambda_m}\left[\LA(n\eps^{-j-1}, n\eps^{-j})\right]\asymp \eps^{\alpha_4+o(1)}.\]
Therefore,
\[\limsup_{K\to \infty}\frac{\log\mu_{\Lambda_m}\left[\LA(n, N)\right]}{\log N}\le\frac{\tilde{C}}{\log(1/\eps)}-\alpha_4,\]
where $\tilde{C}$ is some universal constant. Let $\eps\to 0$, we have 
\[\limsup_{N\to \infty}\frac{\log\mu_{\Lambda_m}\left[\LA(n, N)\right]}{\log N}\le-\alpha_4.\]
We could prove the lower bound similarly:
\[\liminf_{N\to \infty}\frac{\log\mu_{\Lambda_m}\left[\LA(n, N)\right]}{\log N}\ge -\alpha_4.\]
These imply (\ref{eqn::ising_4arms}) and complete the proof. 
\end{proof}

To prove Theorem \ref{thm::ising_boundary}, we will show the proof for $\gamma^+_{2j-1}$, and the results for $\alpha^+_{2j-1}, \gamma^+_{2j}$ can be proved similarly; and 
we will show the proof for $\beta^+_{2j}$, and the results for $\alpha^+_{2j}, \beta^+_{2j-1}$ can be proved similarly. 

\begin{figure}[ht!]
\begin{subfigure}[b]{0.5\textwidth}
\begin{center}
\includegraphics[width=\textwidth]{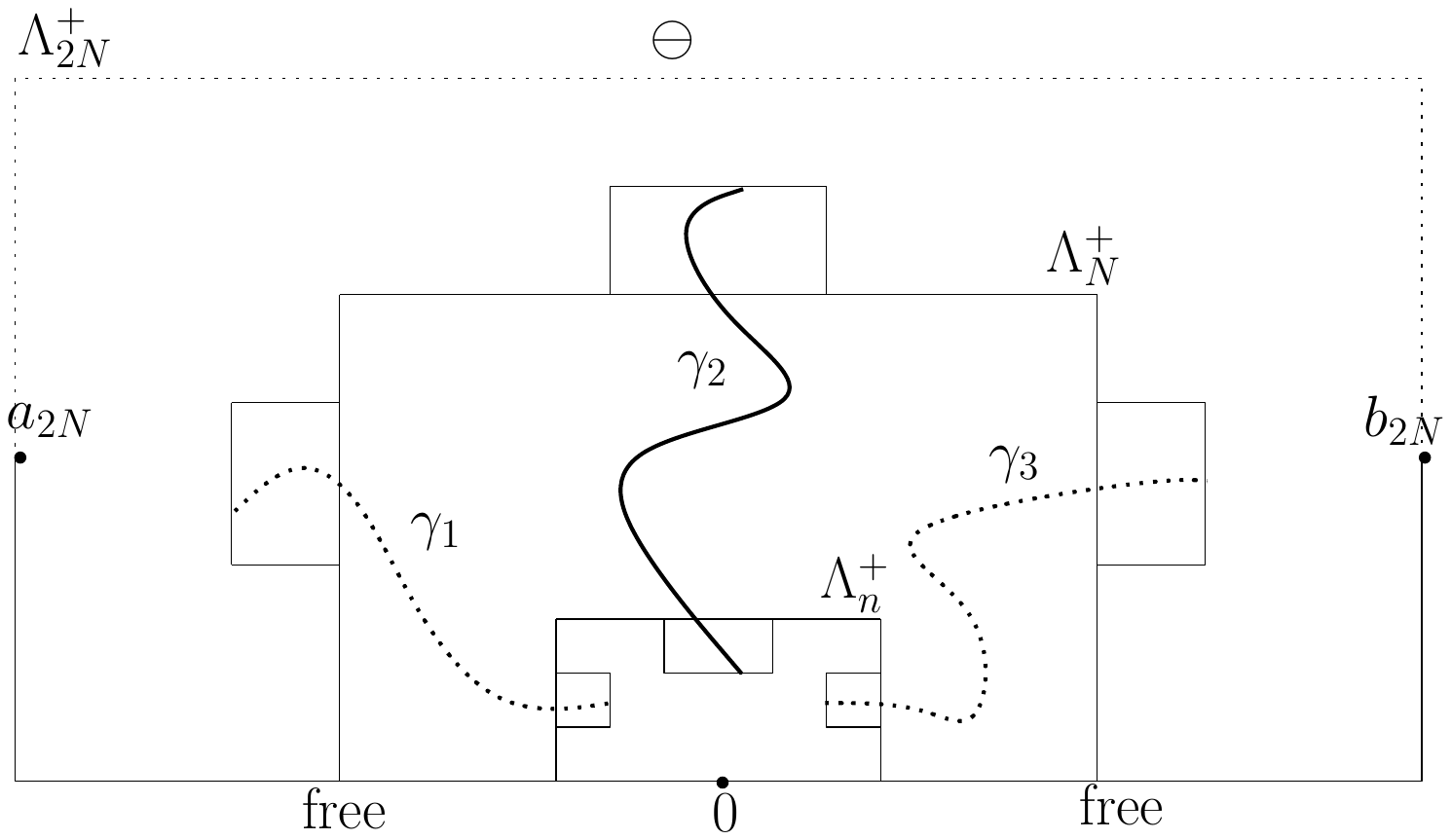}
\end{center}
\caption{$\LA^{^+,I/I}(n,N)$ is the well-separated arm event. }
\end{subfigure}
\begin{subfigure}[b]{0.5\textwidth}
\begin{center}\includegraphics[width=\textwidth]{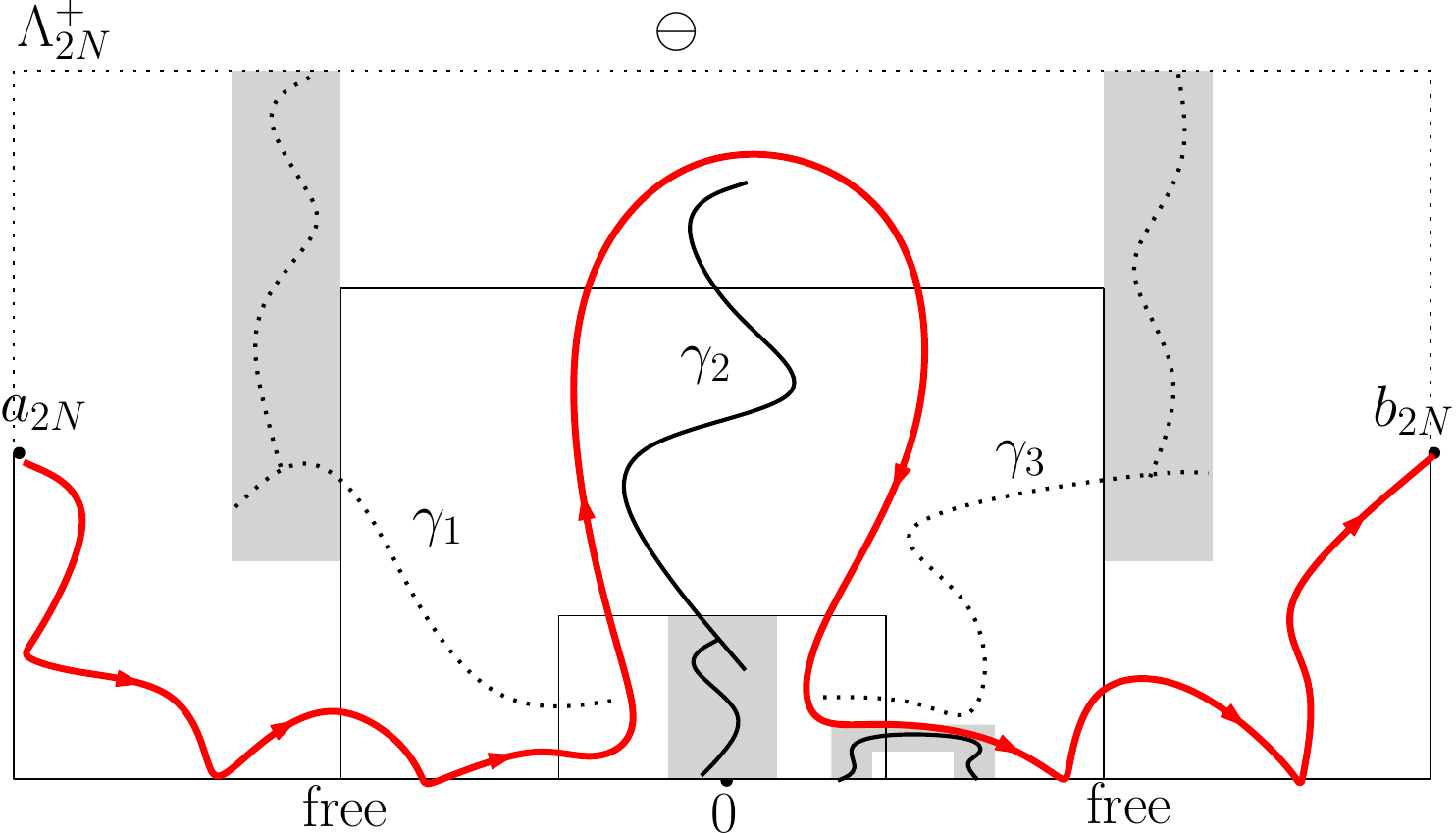}
\end{center}
\caption{The four gray parts are $R_1$ to $R_4$ respectively. }
\end{subfigure}
\caption{\label{fig::ising_gammaodd} The explanation of the proof of (\ref{eqn::ising_gamma_odd}).}
\end{figure}

\begin{proof}[Proof of (\ref{eqn::ising_gamma_odd})]
We will prove the conclusion for $\gamma^+_3$ and the other cases can be proved similarly. Consider $\Lambda_m^+$ with two boundary points $a_m=(-m,m/2)$ and $b_m=(m,m/2)$. Fix the $(\ominus\free)$ boundary condition: the vertices along $\partial\Lambda_m$ from $b_m$ to $a_m$ (counterclockwise) are $\free$ and the vertices from $a_m$ to $b_m$ are $\ominus$. Since we fix $\beta=\beta_c$ and the boundary conditions, and $\omega=(\ominus\oplus\ominus)$, we eliminate them from the notations. We will prove that, for $n<N\le m/2$, 
\begin{equation}\label{eqn::ising_gammaodd}
\mu_{\Lambda^+_m}[\LA^+(n,N)]=N^{-\gamma_3^++o(1)},\quad \text{as }N\to \infty.
\end{equation}

Fix the landing sequence $I=(I_1, I_2, I_3)$ where 
\[I_1=\{-1\}\times[1/2,3/4], \quad I_2=[-1/2, 1/2]\times\{1\},\quad I_3=\{1\}\times [1/2,3/4].\] 
Recall that $\LA^{^+,I/I}(n,N)$ is the $1/8$-well-separated arm events with the landing sequence $nI$ on $\partial \Lambda^+_n$ and $NI$ on $\partial\Lambda^+_N$. The three arms in $\LA^{+, I/I}(n,N)$ are denoted by $(\gamma_1, \gamma_2, \gamma_3)$ where $\gamma_1$ and $\gamma_3$ are $\ominus$ and $\gamma_2$ is $\oplus$.  Consider critical Ising model in $\Lambda^+_{2N}$. Let $R_1$ be the rectangle $[-9N/8, -N]\times[N/2,N]$ and define $\LC_1^{\ominus}$ to be the event that $\gamma_1$ is connected by path of $\ominus$ in $R_1$ to the top of $R_1$. Let $R_2$ be the rectangle $[-3n/4, 3n/4]\times [0,n]$ and define $\LC_2^{\oplus}$ to be the event that $\gamma_2$ is connected by path of $\oplus$ in $R_2$ to the bottom of $R_2$. Let $R_3$ be the rectangle $[N, 9N/8]\times[N/2,N]$ and define $\LC_3^{\ominus}$ to be the event that $\gamma_3$ is connected by path of $\ominus$ in $R_3$ to the top of $R_3$. For $\delta>0$, let $R_4$ be the semi-annulus $[3n/4, 4n]\times[0,n/4]\setminus [n, 3n]\times[0,\delta n]$ and define $\LC_4^{\oplus}(\delta)$ to be the event that there is a path of $\oplus$ in $R_4$ connecting the left bottom to the right bottom.  

By (\ref{eqn::ising_boundary_comparison}), Proposition \ref{prop::ising_rsw} and Corollary \ref{cor::ising_spacemixing}, we could prove, for $\delta>0$ small enough,
\begin{equation}\label{eqn::ising_boundary_gamma_aux1}
\mu_{\Lambda_{2N}}\left[\LA^{+, I/I}(n,N)\right]\asymp\mu_{\Lambda_{2N}}\left[\LA^{+, I/I}(n,N)\cap\LC_1^{\ominus}\cap\LC_2^{\oplus}\cap\LC_3^{\ominus}\cap\LC_4^{\oplus}(\delta)\right],
\end{equation}
where the constants in $\asymp$ are uniform over $n, N$. 

Let $\PP_N$ be the probability measure $\mu_{\Lambda_{2N}}$ where the square lattice is scaled by $1/N$ and let $\PP_{\infty}$ be the law of $\SLE_3(-3/2)$ in $[-2,2]\times[0,2]$ from $(-2,1)$ to $(2,1)$. On the event $\LA^{+, I/I}(n,N)\cap\LC_1^{\ominus}\cap\LC_2^{\oplus}\cap\LC_3^{\ominus}\cap\LC_4^{\oplus}(\delta)$, consider the interface $\eta$ from $a_{2N}$ to $b_{2N}$. Let $\tau_1$ be the first time that $\eta$ hits $\partial\Lambda_n$. The event $\LC_1^{\ominus}$ guarantees that $\eta[0,\tau_1]$ is bounded away from the target $b_{2N}$. The event $\LC_4^{\oplus}(\delta)$ guarantees that $\eta[0,\tau_1]$ is bounded away from the segment $[n,3n]$. The event $\LC_2^{\oplus}$ guarantees that, after $\tau_1$, the interface $\eta$ hits the neighborhood of the point $(0,N)$ at some time $\sigma_1$. The event $\LC_3^{\ominus}$ guarantees that, after $\sigma_1$, the interface $\eta$ hits $\partial\Lambda_n$ again. See Figure \ref{fig::ising_gammaodd}. Therefore, by (\ref{eqn::sle_boundary_gamma_odd}), we have, for $\eps>0$,
\[\limsup_{N\to\infty}\PP_N\left[\LA^{+, I/I}(\eps N,N)\cap\LC_1^{\ominus}\cap\LC_2^{\oplus}\cap\LC_3^{\ominus}\cap\LC_4^{\oplus}(\delta)\right]\lesssim \eps^{\gamma^+_3}\le \liminf_{N\to\infty}\PP_N\left[\LA^+(\eps N, N)\right].\]
Now we can repeat the same proof of Theorem \ref{thm::ising_interior} to obtain (\ref{eqn::ising_gammaodd}).
\end{proof}

\begin{figure}[ht!]
\begin{subfigure}[b]{0.5\textwidth}
\begin{center}
\includegraphics[width=0.8\textwidth]{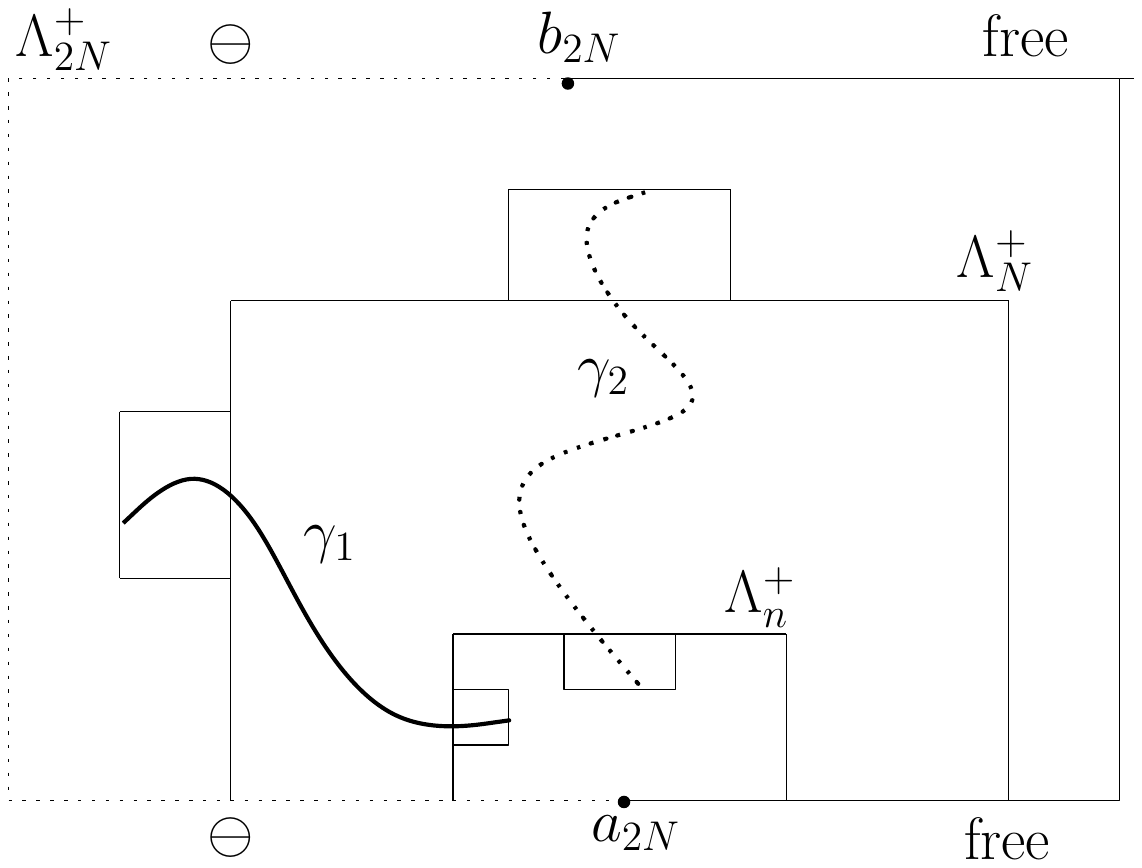}
\end{center}
\caption{$\LA^{^+,I/I}(n,N)$ is the well-separated arm event. }
\end{subfigure}
\begin{subfigure}[b]{0.5\textwidth}
\begin{center}\includegraphics[width=0.8\textwidth]{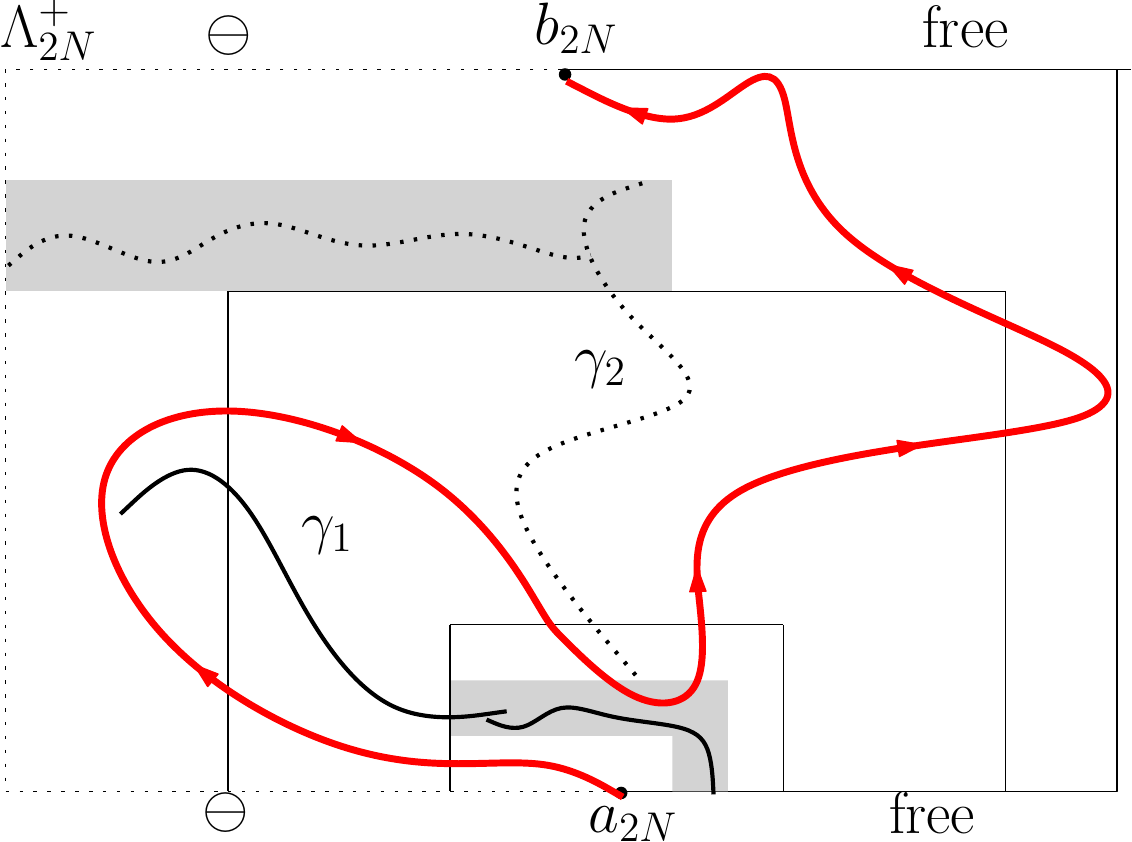}
\end{center}
\caption{The two gray parts are $R_1$ and $R_2$ respectively. }
\end{subfigure}
\caption{\label{fig::ising_betaeven} The explanation of the proof of (\ref{eqn::ising_beta_even}).}
\end{figure}

\begin{proof}[Proof of (\ref{eqn::ising_beta_even})]
We will prove the conclusion for $\beta^+_2$ and the other cases can be proved similarly. Consider $\Lambda_m^+$ with two boundary points $a_m=(0,0)$ and $b_m=(0,m)$. Fix the $(\ominus\free)$ boundary condition: the vertices along $\partial\Lambda_m$ from $b_m$ to $a_m$ (counterclockwise) are $\free$ and the vertices from $a_m$ to $b_m$ are $\ominus$. Since we fix $\beta=\beta_c$ and the boundary conditions, and $\omega=(\oplus\ominus)$, we eliminate them from the notations. We will prove that, for $n<N\le m/2$, 
\begin{equation}\label{eqn::ising_betaeven}
\mu_{\Lambda^+_m}[\LA^+(n,N)]=N^{-\beta_2^++o(1)},\quad \text{as }N\to \infty.
\end{equation}

Fix the landing sequence $I=(I_1, I_2)$ where 
\[I_1=\{-1\}\times[1/2,3/4], \quad I_2=[-1/2, 1/2]\times\{1\}.\] 
Recall that $\LA^{^+,I/I}(n,N)$ is the $1/8$-well-separated arm events with the landing sequence $nI$ on $\partial \Lambda^+_n$ and $NI$ on $\partial\Lambda^+_N$. The two arms in $\LA^{+, I/I}(n,N)$ are denoted by $(\gamma_1, \gamma_2)$ where $\gamma_1$ is $\oplus$ and $\gamma_2$ is $\ominus$.  Consider critical Ising model in $\Lambda^+_{2N}$. Let $R_1$ be the tube $[-n, 3n/4]\times[0, 3n/4]\setminus[-n, n/2]\times[0,n/4]$ and define $\LC_1^{\oplus}$ to be the event that $\gamma_1$ is connected by path of $\oplus$ in $R_1$ to the bottom of $R_1$. Let $R_2$ be the rectangle $[-N, N/2]\times[N/2, 5N/8]$ and define $\LC_2^{\ominus}$ to be the event that $\gamma_2$ is connected by path of $\ominus$ in $R_2$ to the left side of $R_2$. 

By (\ref{eqn::ising_boundary_comparison}), Proposition \ref{prop::ising_rsw} and Corollary \ref{cor::ising_spacemixing}, we could prove, for $\delta>0$ small enough,
\begin{equation}\label{eqn::ising_boundary_beta_aux1}
\mu_{\Lambda_{2N}}\left[\LA^{+, I/I}(n,N)\right]\asymp\mu_{\Lambda_{2N}}\left[\LA^{+, I/I}(n,N)\cap\LC_1^{\oplus}\cap\LC_2^{\ominus}\right],
\end{equation}
where the constants in $\asymp$ are uniform over $n, N$. 

Let $\PP_N$ be the probability measure $\mu_{\Lambda_{2N}}$ where the square lattice is scaled by $1/N$ and let $\PP_{\infty}$ be the law of $\SLE_3(-3/2)$ in $[-2,2]\times[0,2]$ from $(0,0)$ to $(0,1)$. On the event $\LA^{+, I/I}(n,N)\cap\LC_1^{\oplus}\cap\LC_2^{\ominus}$, consider the interface $\eta$ from $a_{2N}$ to $b_{2N}$, the event guarantees that the interface hits the neighborhood of the point $(-N, N/2)$, and then comes back to $\Lambda_n^+$.  
See Figure \ref{fig::ising_betaeven}. Therefore, by (\ref{eqn::sle_boundary_alpha_even_upper}) and (\ref{eqn::sle_boundary_alpha_even_lower}) (taking $\rho=-3/2$), we have, for $\eps>0$,
\[\limsup_{N\to\infty}\PP_N\left[\LA^{+, I/I}(n,N)\cap\LC_1^{\oplus}\cap\LC_2^{\ominus}\right]\lesssim \eps^{\beta_2^+}\le \liminf_{N\to\infty}\PP_N\left[\LA^+(\eps N, N)\right].\]
Now we can repeat the same proof of Theorem \ref{thm::ising_interior} to obtain (\ref{eqn::ising_betaeven}).
\end{proof}

\begin{remark}
Consider Propositions \ref{prop::sle_boundary_alpha} and \ref{prop::sle_boundary_beta}. Suppose $\eta$ is the interface of critical Ising model, then, taking $\kappa=3, \rho=\kappa/2-3$ in Proposition \ref{prop::sle_boundary_alpha}, we know that $\alpha^+_j$ is the arm exponents for the boundary conditions $(\ominus\free)$; moreover, taking $\kappa=3, \rho=-3/2$ in Proposition \ref{prop::sle_boundary_beta}, we have that $\beta^+_j$ should also be the arm exponents for the boundary conditions $(\ominus\free)$. Indeed, we have that these two formulae are the same if and only if $\rho=\kappa/2-3$. This is consistent with what we expect from the critical Ising model. 
\end{remark}

%% file: tex/appendix_onepoint_sle.tex
\begin{theorem}\label{thm::onepoint_sle_multiple}
Suppose $\eta$ is an $\SLE_{\kappa}(\rho^L;\rho^{1,R}, \rho^{2,R})$ process with force points $(x^L; x^R, 1)$ where $x^L\le 0$ and $x^R\in [0,1)$. 
Assume that 
\begin{equation}\label{eqn::assumptions_rho}
\kappa>0, \quad \rho^L>-2, \quad \rho^{1,R}>(-2)\vee(\kappa/2-4), \quad \rho^{1,R}+\rho^{2,R}>\kappa/2-4.
\end{equation}
Define 
\[\alpha=(\rho^{1,R}+2)(\rho^{1,R}+\rho^{2,R}+4-\kappa/2)/\kappa,\quad \beta=2(\rho^{1,R}+\rho^{2,R}+4-\kappa/2)/\kappa.\]
Define, for $\eps\in (0,1/2)$ and $r\ge 4$,
\[\tau_{\eps}=\inf\{t: \eta(t)\in \partial B(1, \eps(1-x^R))\},\quad S_r=\inf\{t: \eta(t)\in \partial B(0,r)\}.\]
Then we have
\[\eps^{\alpha}\left(1-x^R\right)^{\beta}\lesssim \PP[\tau_{\eps}\le S_r]=\eps^{\alpha+o(1)},\]
where the constant in $\lesssim$ depends only on $\kappa, \rho^L, \rho^{1,R}, \rho^{2,R}$ and the $o(1)$ term goes to zero as $\eps\to 0$ which depends only on $\kappa, \rho^L, \rho^{1,R}, \rho^{2,R}$ and $x^R, r$.  
\end{theorem}

\begin{corollary}\label{cor::onepoint_sle_multiple}
Assume the same notations as in Theorem \ref{thm::onepoint_sle_multiple}.  
Assume that 
\begin{equation}
\kappa>0, \quad \rho^L\in (-2,0], \quad \rho^{1,R}>(-2)\vee(\kappa/2-4), \quad \rho^{1,R}+\rho^{2,R}>\kappa/2-4.
\end{equation}
Then we have
\[\eps^{\alpha}\left(1-x^R\right)^{\beta}\lesssim \PP[\tau_{\eps}<\infty]=\eps^{\alpha+o(1)},\]
where the constant in $\lesssim$ depends only on $\kappa, \rho^L, \rho^{1,R}, \rho^{2,R}$ and the $o(1)$ term goes to zero as $\eps\to 0$ which depends only on $\kappa, \rho^L, \rho^{1,R}, \rho^{2,R}$ and $x^R$.  
\end{corollary}
We also expect that Corollary \ref{cor::onepoint_sle_multiple} holds for all $\rho^L>-2$, but we do not have a proof yet for $\rho^L\ge 0$. 
Before proving the theorem, we first summarize the existing related results. For standard $\SLE_{\kappa}$ with $\kappa\in (0,8)$, a stronger conclusion is known \cite{AlbertsKozdronIntersectionProbaSLEBoundary}:
\[\PP[\eta\text{ hits }B(1,\eps)]\asymp \eps^{\alpha},\quad \alpha=(8-\kappa)/\kappa.\] 
For $\SLE_{\kappa}(\rho)$ with one force point at $x^R\in [0,1)$, a stronger conclusion is known \cite[Proposition 5.4]{LawlerMinkowskiSLERealLine}:
\[\PP[\tau_{\eps}<\infty]\asymp \eps^{\alpha}(1-x^R)^{\beta},\quad \alpha=(\rho+2)(\rho+4-\kappa/2)/\kappa, \quad \beta=2(\rho+4-\kappa/2)/\kappa.\]
For $\SLE_{\kappa}(\rho^{1,R},\rho^{2,R})$ process, the conclusion in Theorem \ref{thm::onepoint_sle_multiple} is proved in \cite[Theorem 3.1]{MillerWuSLEIntersection}.
\begin{lemma}\label{lem::onepoint_sle_multiple}
Assume the same notations as in Theorem \ref{thm::onepoint_sle_multiple}. For $\delta\in (0,1/4)$ and $r\ge 4$, we have 
\[\eps^{\alpha}(1-x^R)^{\beta}\lesssim \PP[\tau_{\eps}\le S_r, \Im{\eta(\tau_{\eps})}\ge \delta\eps (1-x^R)]\lesssim \eps^{\alpha}(1-x^R)^{\beta}\delta^{-\beta}r^B,\]
where $B=0\vee (\beta\rho^L/2)$ and the constants in $\lesssim$ are uniform over $\eps, \delta, x^L, x^R, r$. 
\end{lemma}
\begin{proof}
Let $V_t^L$ be the evolution of $x^L$ and $V_t^R$ be the evolution of $x^R$. 
Set
\[M_t=g_t'(1)^{\nu(\nu+2\rho^{2,R}+4-\kappa)/(4\kappa)}(g_t(1)-W_t)^{\nu/\kappa}\left(\frac{g_t(1)-V_t^R}{1-x^R}\right)^{\nu\rho^{1,R}/(2\kappa)}\left(\frac{g_t(1)-V_t^L}{1-x^L}\right)^{\nu\rho^{L}/(2\kappa)},\]
where $\nu=-\beta\kappa\le 0$. From \cite[Theorem 6]{SchrammWilsonSLECoordinatechanges}, we know that $M$ is a local martingale and the law of $\eta$ weighted by $M$ becomes the law of $\SLE_{\kappa}(\rho^L;\rho^{1,R}, \rho^{2,R}+\nu)$ with force points $(x^L; x^R, 1)$.  

On the event $E_{\eps}(\delta, r):=\{\tau_{\eps}\le S_r, \Im{\eta(\tau_{\eps})}\ge \delta\eps (1-x^R)\}$, let us estimate the terms in $M_t$ one by one for $t=\tau_{\eps}$. Let $O_t$ be the image of the rightmost point of $\eta[0,t]\cap \R$ under $g_t$. By Koebe 1/4 theorem, we know that $g_t(1)-O_t\asymp g_t'(1)\eps (1-x^R)$. 
\begin{itemize}
\item Consider the term $g_t(1)-W_t$. Since $\Im{\eta(t)}\ge \delta\eps(1-x^R)$, combining with \cite[Lemma 3.4]{MillerWuSLEIntersection}, we have 
\[g_t'(1)\eps (1-x^R)\asymp g_t(1)-O_t\le g_t(1)-W_t\lesssim (g_t(1)-O_t)/\delta\asymp g_t'(1)\eps (1-x^R)/\delta.\]
\item Consider the term $g_t(1)-V^R_t$. If $x^R$ is swallowed by $\eta[0,t]$, then we have $g_t(1)-V^R_t=g_t(1)-O_t\asymp g_t'(1)\eps (1-x^R)$. If not, by Keobe 1/4 theorem, we have $g_t(1)-V^R_t\asymp g_t'(1)\eps(1-x^R)$. In any case, we have
\[g_t(1)-V^R_t\asymp g_t'(1)\eps(1-x^R).\]
\item Consider the term $g_t(1)-V_t^L$. Since $g_t(1)-V_t^L$ is increasing in $t$, we have $g_t(1)-V_t^L\ge 1-x^L$ for all $t$. Suppose $\LB^{yi}$ is a Brownian motion starting from $yi$, from \cite[Remark 3.50]{LawlerConformallyInvariantProcesses}, we know that 
\[g_t(1)-V^L_t=\lim_{y\to\infty}\pi y\PP[\LB^{yi} \text{ exits } \HH\setminus \eta[0,t] \text{ through the union } [x^L, 0]\cup\eta[0,t]\cup[0,1]].\]
Since $t=\tau_{\eps}\le S_r$, we have 
\[g_t(1)-V^L_t\le \lim_{y\to\infty}\pi y\PP[\LB^{yi} \text{ exits } \HH\setminus B(0,r) \text{ through the union } [x^L, 0]\cup B(0,r)].\]
If $|x^L|\le r$, then $g_t(1)-V^L_t\le 4r$; if $|x^L|\ge r$, we have $g_t(1)-V_t^L\le |x^L|+3r$. Thus we have 
\[1\le \frac{g_t(1)-V_t^L}{1-x^L}\le 4r.\]
\end{itemize} 
Combining the above three parts, on the event $E_{\eps}(\delta, r)$,  we have 
\[\eps^{-\alpha}(1-x^R)^{-\beta}\delta^{-\beta}r^{-B}\lesssim M_{\tau_{\eps}}\lesssim \eps^{-\alpha}(1-x^R)^{-\beta}r^{0\vee(-\beta\rho^L/2)}.\]
Therefore, we have the lower bound:
\[\PP[E_{\eps}(\delta, r)]\ge \PP[E_{\eps}(4,1/4)]\gtrsim \eps^{\alpha}(1-x^R)^{\beta}\E[M_{\tau_{\eps}}1_{E_{\eps}(4,1/4)}]=\eps^{\alpha}(1-x^R)^{\beta}\PP^*[E^*_{\eps}(4,1/4)],\]
where $\eta^*$ is an $\SLE_{\kappa}(\rho^L; \rho^{1,R}, \rho^{2,R}+\nu)$ with force points $(x^L; x^R, 1)$ and $\PP^*$ is its law, and the event $E^*_{\eps}(r,\delta)$ is defined for $\eta^*$. Note that 
\[\rho^{1,R}+\rho^{2,R}+\nu=\kappa-8-\rho^{1,R}-\rho^{2,R}<\kappa/2-4.\]
Thus $\eta^*$ accumulates at the point $1$ at finite time. Let $\phi(z)=z/(1-z)$ be the Mobius transform of $\HH$ that sends $(0,1,\infty)$ to $(0,\infty,-1)$ and let $\hat{\eta}$ be the image of $\eta^*$ under $\phi$. Then $\hat{\eta}$ is an $\SLE_{\kappa}(\rho^{1,R}+\rho^{2,R}+2-\rho^L, \rho^L; \rho^{1,R})$ with force points $(-1, \phi(x^L);\phi(x^R))$. Define $\hat{E}$ to be the event that $\hat{\eta}$ never hits $B(-1,1/3)$ and $\hat{\eta}$ exits the ball of radius $1/(\eps(1-x^R))$ through the angle interval $[\pi/4,3\pi/4]$. It is clear that $\PP^*[E^*_{\eps}(4,1/4)]\ge \hat{\PP}[\hat{E}] \asymp 1$ (see for instance \cite[Lemma 2.3]{MillerWuSLEIntersection}), since $\rho^{1,R}+\rho^{2,R}+2>\kappa/2-2$. This gives the lower bound. 

For the upper bound, we have
\[\PP[E_{\eps}(\delta, r)]\lesssim \eps^{\alpha}(1-x^R)^{\beta}\delta^{-\beta}r^B\E[M_{\tau_{\eps}}1_{E_{\eps}(\delta, r)}]\le \eps^{\alpha}(1-x^R)^{\beta}\delta^{-\beta}r^B,\]
as desired. 
\end{proof}

\begin{lemma}\label{lem::rareevent_sle_multiple}
Suppose that $\eta$ is an $\SLE_{\kappa}(\rho^L; \rho^{1,R}, \rho^{2,R})$ with force points $(x^L; x^{1,R}, x^{2,R})$ where $x^L\le 0$ and $0\le x^{1,R}\le 1$ and $x^{1,R}\le x^{2,R}$. Assume (\ref{eqn::assumptions_rho}) holds. Then there exists a function $p(\delta)\to 0$ as $\delta\to 0$ which depends only on $\kappa, \rho^L, \rho^{1,R}, \rho^{2,R}$ such that 
\[\PP[\eta \text{ hits }B(1,\delta)]\le p(\delta).\]
We emphasize that $p(\delta)$ is uniform over $x^L\le 0\le x^{1,R}\le 1$ and $x^{2,R}\ge x^{1,R}$. 
\end{lemma}
\begin{proof}
Define $f(x^L, x^{1,R}, x^{2,R}, \delta)=\PP[\eta \text{ hits }B(1,\delta)]$. From \cite[Section 4.7]{LawlerConformallyInvariantProcesses}, we know that the function $f$ is continuous. Since (\ref{eqn::assumptions_rho}) holds, we know that $f(x^L, x^{1,R}, x^{2,R}, \delta)\to 0$ as $\delta\to 0$. When $|x^L|, x^{2,R}\to \infty$, the law of $\eta$ becomes the law of $\SLE_{\kappa}(\rho^{1,R})$, thus the function 
\[\lim_{|x^L|, x^{2,R}\to \infty}f(x^L, x^{1,R}, x^{2,R}, \delta)\]
goes to zero as $\delta\to 0$. This implies 
\[p(\delta):=\sup f(x^L, x^{1,R}, x^{2,R}, \delta)\to 0,\quad \text{as }\delta\to 0,\] 
where the supreme is taken over $x^L\le 0, x^{1,R}\in [0,1], x^{2,R}\ge x^{1,R}$.
\end{proof}
\begin{proof}[Proof of Theorem \ref{thm::onepoint_sle_multiple}]
Lemma \ref{lem::onepoint_sle_multiple} gives the lower bound, and we only need to show the upper bound. Pick an integer $n$ such that $ 2^n\eps\le 1/4$. For $1\le k\le n$, let $T_k$ be the first time that $\eta$ hits the ball centered at $1$ with radius $2^{n+1-k}\eps(1-x^R)$. Define 
\[\LF_k=\{\Im{\eta(T_k)}\ge \delta 2^{n+1-k}\eps(1-x^R)\}.\]
By Lemma \ref{lem::onepoint_sle_multiple}, we know that 
\begin{align*}
\PP[\tau_{\eps}\le S_r]&\le \sum_1^n \PP[\{T_k<\infty\}\cap \LF_k]+\PP[\{\tau_{\eps}\le S_r\}\cap_1^n\LF_k^c]\\
&\lesssim 2^{n\alpha}\eps^{\alpha}(1-x^R)^{\beta}\delta^{-\beta}r^B+\PP[\{\tau_{\eps}\le S_r\}\cap_1^n\LF_k^c].
\end{align*}
Consider the event $\{\tau_{\eps}\le S_r\}\cap_1^n\LF_k^c$. For $0\le k\le n-4$, given $\eta[0,T_k]$, we will estimate the conditional probability that $\{T_{k+4}<\infty\}\cap\LF_{k+4}^c$. Denote $2^{n-k-3}\eps(1-x^R)$ by $u$. The event $\LF_{k+1}^c$ implies that $\eta$ hits $B(1,u)$ through the union of the following two balls $B(1-u, \delta u)\cup B(1+u, \delta u)$. Denote $g_{T_k}-W_{T_k}$ by $f$. The image of $\eta(t+T_k)$ under $f$ is still an $\SLE_{\kappa}(\rho^L; \rho^{1,R}, \rho^{2,R})$ process. Since the distance between $\eta[0,T_k]$ and $1$ is $16u$, by Lemma \ref{lem::image_insideball} we know that $f(B(1-u,\delta u)\cup B(1+u, \delta u))$ is contained in 
\[B(f(1-u), 4\delta u f'(1-u))\cup B(f(1+u, 4\delta u f'(1+u))).\]
By Koebe 1/4 theorem, we know that 
\[f(1-u)\gtrsim u f'(1-u),\quad f(1+u)\gtrsim uf'(1+u).\]
Thus, by Lemma \ref{lem::rareevent_sle_multiple}, we know that 
\[\PP[\{T_{k+4}<\infty\}\cap \LF_{k+4}^c\cond \eta[0,T_k]]\le p(C\delta),\]
for some universal constant $C$. Therefore, there is a function $q(\delta)\to 0$ as $\delta\to 0$ which depends only on $\kappa, \rho^L, \rho^{1,R}$ and $\rho^{2,R}$ such that 
\[\PP[\{\tau_{\eps}\le S_r\}\cap_1^n\LF_k^c]\le q(\delta)^n.\]
Thus we have 
\[\PP[\tau_{\eps}\le S_r]\lesssim 2^{n\alpha}\eps^{\alpha}(1-x^R)^{\beta}\delta^{-\beta}r^B+q(\delta)^n.\]
This implies the conclusion.
\end{proof}
\begin{proof}[Proof of Corollary \ref{cor::onepoint_sle_multiple}]
Assume the same notations as in the proof of Theorem \ref{thm::onepoint_sle_multiple}. When $\rho^L\le 0$, we have $B=0$, thus 
\[\PP[\tau_{\eps}<\infty]\lesssim 2^{n\alpha}\eps^{\alpha}(1-x^R)^{\beta}\delta^{-\beta}+q(\delta)^n.\]
This implies the conclusion.
\end{proof}